\documentclass{amsart}
\usepackage{amssymb,stmaryrd}
\usepackage{enumerate,enumitem}
\usepackage[colorlinks,citecolor=blue,urlcolor=black,linkcolor=black]{hyperref}

\newtheorem{theorem}{Theorem}[section] 
\newtheorem{claim}{Claim}[theorem]
\newtheorem{lemma}[theorem]{Lemma} 
\newtheorem{observation}[theorem]{Observation} 
\newtheorem{corollary}[theorem]{Corollary} 
 
\newtheorem{mainthm}{Theorem}

\theoremstyle{definition}
\newtheorem{definition}[theorem]{Definition}
\newtheorem{notation}[theorem]{Notation}

\theoremstyle{remark}
\newtheorem{remark}[theorem]{Remark}

\newenvironment{PROOF}[2][\proofname.]
{\begin{proof}[#1]\renewcommand{\qedsymbol}{\textsquare$_{\rm #2}$}}
	{\end{proof}}

\DeclareMathOperator{\sym}{Sym}
\DeclareMathOperator{\fs}{FS}

\DeclareMathOperator{\otp}{otp}
\DeclareMathOperator{\dom}{dom}
\DeclareMathOperator{\ran}{Im}
\DeclareMathOperator{\acc}{acc}
\DeclareMathOperator{\pr}{Pr}
\DeclareMathOperator{\U}{U}

\newcommand{\one}{\mathop{1\hskip-3pt {\rm l}}} 
\newcommand{\diagonal}{\bigtriangleup}
\newcommand{\rest}{\mathbin\upharpoonright}
\newcommand{\mal}{\mathrel{\leq_{\rm m}}}
\newcommand{\s}{\subseteq}
\newcommand{\br}{\blacktriangleright}
\newcommand{\varp}{{\varepsilon}}
\newcommand{\tieconcat}{{}^\smallfrown}

\newcommand*\axiomfont[1]{\textsf{\textup{#1}}}
\newcommand\zfc{\axiomfont{ZFC}}

\newcommand\ch{\axiomfont{CH}}

\title {A Shelah group in ZFC}
\author {Márk Poór}
\address{Einstein Institute of Mathematics\\
Edmond J. Safra Campus, Givat Ram\\
The Hebrew University of Jerusalem\\
Jerusalem, 91904, Israel\\}
\email{sokmark@gmail.com}

\author {Assaf Rinot}
\address{Department of Mathematics, Bar-Ilan University, Ramat-Gan 5290002, Israel.}
\urladdr{http://www.assafrinot.com}
\email{rinotas@math.biu.ac.il}


\begin{document}
	\begin{abstract}
	In a paper from 1980,
	Shelah constructed an uncountable group all of whose proper subgroups are countable.
	Assuming the continuum hypothesis, he constructed an uncountable group $G$ that moreover admits
	an integer $n$ satisfying that for every uncountable $X\subseteq G$,
	every element of $G$ may be written as a group word of length $n$ in the elements of $X$.
The former is called a \emph{J{\'o}nsson group} and the latter is called a \emph{Shelah group}.
	
	In this paper, we construct a Shelah group on the grounds of $\zfc$ alone, that is,
	without assuming the continuum hypothesis. More generally, 
	we identify a combinatorial condition (coming from the theories of negative square-bracket partition relations and strongly unbounded subadditive maps) sufficient for the 
	construction of a Shelah group of size $\kappa$, and prove that the condition holds true for all successors of regular cardinals (such as $\kappa=\aleph_1,\aleph_2,\aleph_3,\ldots$).
This also yields the first consistent example of a Shelah group of size a limit cardinal.
	\end{abstract}
	
\date{This is a preliminary preprint as of May 18, 2023. The most updated version may be found in \textsf{http://p.assafrinot.com/60}.}
	\maketitle

\numberwithin{equation}{section}

\section{Introduction}
For a prime number $p$, the {Pr\"ufer $p$-group}
$$\{x\in \mathbb C\mid \exists n\in\mathbb N\,(x^{p^n}=1)\}$$ 
is an example of an infinite subgroup of $(\mathbb C,\cdot)$ all of whose proper subgroups are finite.
In \cite{MR571100}, Ol'\v{s}anski\u{\i} constructed the so-called \emph{Tarski monsters}
that are in particular infinite countable groups all of whose proper subgroups are finite.
Then, in \cite{Sh:69}, answering a question of Kurosh,
Shelah constructed an uncountable group all of whose proper subgroups are countable.
All of those are examples of so-called \emph{J{\'o}nsson groups}, i.e., 
an infinite group $G$ having no proper subgroups of full size.
An even more striking concept is that of a  \emph{boundedly-J{\'o}nsson group}, that is,
a group $G$ admitting a positive integer $n$ such that for every $X\s G$ of full size,
it is the case that $X^n=G$, i.e., every element of $G$ may be written as a group word of length exactly $n$ in the elements of $X$.
In \cite{Sh:69}, Shelah constructed a boundedly-J{\'o}nsson group of size $\aleph_1$
with the aid of Continuum Hypothesis ($\ch$). More generally, Shelah proved that
$2^\lambda=\lambda^+$ yields a boundedly-J{\'o}nsson group of size $\lambda^+$.
By now, the concept of boundedly-J{\'o}nsson groups is named after him:
\begin{definition}\label{nshelah}  A group $G$ is \emph{$n$-Shelah} 
if $X^n=G$ for every $X\s G$ of full size.

A group is \emph{Shelah} if it is $n$-Shelah for some positive integer $n$.
\end{definition}

Along the years, variations of this concept were studied quite intensively, and from various angles.
A group $G$ is said to be \emph{Cayley bounded} with respect to a subset $S\s G$
if there exists a positive integer $n_S$ such that $G=\bigcup_{i=1}^{n_S}(S\cup S^{-1})^i$,
i.e., every element of $G$ may be written as a group word of length at most $n_S$ in the elements of $S$ and inverses of elements of $S$.
Extending the work of Macpherson and Neumann \cite{macpherson1990subgroups},
Bergman proved \cite{MR2239037} that the permutation group $\sym(\Omega)$ of an infinite set $\Omega$ 
is Cayley bounded with respect to all of its generating sets.
Soon after, the notion \emph{Bergman property} was coined as the assertion of being Cayley bounded with respect to all generating sets.
Since then it has received a lot of attention, see \cite{MR2122432, MR2154425,  MR2266526, MR2241973, MR2332091, MR2418802,  MR2645225, MR2520386, MR2959420, MR2982769}.
More recent examples include the work of Dowerk \cite{MR4095507} on von Neumann algebras with unitary groups possessing the property of \emph{$n$-strong uncountable cofinality} (i.e.\ having a common Cayley bound $n$ for all generating sets, 
and the group is not the union of an infinite countable strictly increasing sequence of subgroups),
and Shelah's work on locally finite groups \cite{MR4186458}.
It is worth mentioning that the notion of strong uncountable cofinality has also geometric reformulations, e.g, by Cornulier \cite{cornulier2006strongly}, Pestov (see \cite[Theorem 1.2]{MR2503307}) and Rosendal \cite[Proposition 3.3]{MR2503307}.

\medskip

Shelah's 1980 construction from $\ch$ was of a $6640$-Shelah group.
It left open two independent questions:
\begin{enumerate}
\item Can $\ch$ be used to construct an $n$-Shelah group for a small number of $n$? 
\item  Is $\ch$ necessary for the construction of an $n$-Shelah group?
\end{enumerate}

Recently, in \cite{banakh2022nonpolybounded}, Banakh addressed the first question,
using $\ch$ to construct a $36$-Shelah group. 
Even more recently, Corson, Ol'\v{s}anski\u{\i} and Varghese \cite{corson2023steep} addressed the second question,
constructing the first $\zfc$ example of a J{\'o}nsson group of size $\aleph_1$
to have the Bergman property.
Unfortunately, the new example stops short from being Shelah,
as every generating set $S$ of this group has its own $n_S$.
In this paper, an affirmative answer to the second question is finally given,
where a Shelah group of size $\aleph_1$ is constructed within $\zfc$. Specifically:
\begin{mainthm}\label{thma} For every infinite regular cardinal $\lambda$, there exists a $10120$-Shelah group of size $\lambda^+$.
In particular, there exist Shelah groups of size $\aleph_1,\aleph_2,\aleph_3,\ldots$.
\end{mainthm}

The proof of Theorem~\ref{thma} reflects advances both in small cancellation theory and in the study of infinite Ramsey theory.
Towards it,
we prove a far-reaching extension
of Hesse's amalgamation lemma,
and we obtain two maps, one coming from the theory of negative square-bracket partition relations,
the other coming from the theory of strongly subadditive functions,
and the two maps have the property that they may be triggered simultaneously,
making them `active' over each other.

The connection to infinite Ramsey theory should not come as a surprise.
First, note that an $n$-Shelah group of size $\aleph_0$ does not exist,
since such a group would have induced a coloring
$c:[\mathbb N]^n\rightarrow k$ for a large enough integer $k$
admitting no infinite homogeneous set,\footnote{See the proof of Corollary~\ref{cor514}.}
in particular contradicting Ramsey's theorem $\aleph_0\rightarrow(\aleph_0)^n_k$.
A deeper connection to (additive) Ramsey theory is 
in the fact that the existence of J{\'o}nsson group of size $\kappa$
is equivalent to a very strong failure of the higher analog of Hindman's finite sums theorem \cite{MR0349574}.
Indeed, by \cite[Corollary~2.8]{paper27}, 
if there exists a J{\'o}nsson group of size $\kappa$, 
then for \emph{every} Abelian group $G$ of size $\kappa$,
there exists a map $c:G\rightarrow G$ such that for every $X\s G$ of full size,
$c\restriction\fs(X)$ is onto $G$, i.e.,
$$\{ c(x_1+\cdots+x_n)\mid n\in\mathbb N,~\{x_1,\ldots,x_n\}\in[X]^n\}=G.$$
Conversely, if $G$ is an Abelian group of size $\kappa$
admitting a map $c:G\rightarrow G$ as above, then the structure $(G,+,c)$
is easily an example of a so-called \emph{J{\'o}nsson algebra} \cite{MR0345895}
of size $\kappa$, which by Corson's work \cite{MR4428866} implies the existence of a J{\'o}nsson group of size $\kappa$.

The fact that the elimination of $\ch$ goes through advances in the theory of partition calculus of uncountable cardinals 
should not come as a surprise, either. To give just one example,
we mention that that three decades after 
Juh\'{a}sz and Hajnal \cite{MR336705} constructed an $L$-space  with the aid of $\ch$,
Moore \cite{Moore} gave a $\zfc$
construction of an $L$-space by establishing a new unbalanced partition relation for the first uncountable cardinal.

\medskip

Having discussed Shelah groups of size $\aleph_0$ and of size a successor cardinal, the next question is whether 
it is possible to construct a Shelah group of size an uncountable \emph{limit} cardinal. 
To compare, a natural ingredient available for transfinite constructions of length 
a successor cardinal $\kappa=\lambda^+$ 
is the existence of $\lambda$-filtrations of all ordinals less than $\kappa$.
We overcome this obstruction at the level of a limit cardinal $\kappa$ by employing subadditive strongly unbounded maps 
$e:[\kappa]^2\rightarrow\lambda$ having arbitrarily large gaps between $\lambda$ and $\kappa$.
This way, we obtain the first consistent example of a Shelah group of size a limit cardinal. More generally:

\begin{mainthm}
For every regular uncountable cardinal $\kappa$ satisfying the combinatorial principle $\square(\kappa)$, there exists a Shelah group of size $\kappa$.
\end{mainthm}

By a seminal work of Jensen \cite{jensen72},
in G\"odel's constructible universe \cite{MR2514}, the combinatorial principle $\square(\kappa)$ holds for every regular uncountable cardinal $\kappa$ that is not weakly compact.
As the reader may anticipate, a cardinal $\kappa$ is \emph{weakly compact} if it is a regular limit cardinal satisfying the higher analog of Ramsey's theorem $\kappa\rightarrow(\kappa)^2_2$ .
Altogether, we arrive at the following optimal result:

\begin{mainthm} In G\"odel's constructible universe, for every regular uncountable cardinal $\kappa$, the following are equivalent:
\begin{itemize}
\item There exists a Shelah group of size $\kappa$;
\item Ramsey's partition relation $\kappa\rightarrow(\kappa)^2_2$ fails.
\end{itemize}
\end{mainthm}

We conclude the introduction by discussing additional features that the groups constructed here possesses.
A group is said to be \emph{topologizable} if it admits a nondiscrete Hausdorff group topology; otherwise, it is \emph{nontopologizable}. 
The first consistent instance for a nontopologizable group was the group constructed by Shelah in \cite{Sh:69} using $\ch$.
Shortly after, an uncountable $\zfc$  example was given by Hesse \cite{hesse1979topologisierbarkeit}. Then a countable such group was given by Ol'\v{s}anski\u{\i}  \cite[Theorem~31.5]{ol2012geometry} (an account of his construction may be found in \cite[\S13.4]{adian2006classifications}). Ol'\v{s}anski\u{\i}'s group is periodic; a torsion-free example  was given by Klyachko and Trofimov in \cite{KlyachkoTrofimov}.

The group we construct in this paper is nontopologizable,
which follows combining the property of Shelah-ness together with the fact that there will be a filtration of the group consisting of malnormal subgroups.
Moreover, our group contains a nonalgebraic unconditionally closed set, which can be shown by proving that small sets can be covered by a topologizable  subgroup, similarly to the argument by Sipacheva \cite[Lemmas 1 and A.4]{sipacheva2006consistent}.

\subsection{Organization of this paper}
In Section~\ref{sec2}, 
we fix our notations and conventions,
and provide some necessary background from small cancellation theory,

In Section~\ref{pumpupsection},
we prove an amalgamation lemma that will serve as a building block in our recursive construction in Section~\ref{sec4}.

In Section~\ref{infcombinatorics}, we provide set-theoretic sufficient conditions for 
the existence of two types of maps to exist, and moreover be active over each other.
The first type comes from the classical theory of negative square-bracket partition relations \cite[\S18]{MR202613},
and enables to eliminate the need for $\ch$ in the construction of a Shelah group of size $\aleph_1$. 
The second type comes from the theory of subadditive strongly unbounded functions \cite{paper36},
and enables to push the construction to higher cardinals including limit cardinals.
At the level of successors of regulars, both of these colorings are obtained in $\zfc$ using the method of \emph{walks on ordinals} \cite{todorcevic_book} that did not exist at the time Shelah's paper \cite{Sh:69} was written.

In Section~\ref{sec4}, we provide a transfinite construction of Shelah groups guided by the colorings given by Section~\ref{infcombinatorics}.

\section{Preliminaries}\label{sec2}

\subsection{Notations and conventions}
	Under ordinals we always mean von Neumann ordinals, and for a set $X$ the symbol $|X|$ always refers to the smallest ordinal with the same cardinality.  For a set $X$ the symbol $\mathcal P(X)$ denotes the power set of $X$, while if $\theta$ is an cardinal we use the standard notation $[X]^\theta$ for $\{Y \in \mathcal P(X)\mid |Y| = \theta\}$, similarly for $[X]^{<\theta}$, $[X]^{\leq \theta}$, etc.~By a sequence we mean a function on an ordinal, where for a sequence $ \overline{s}= \langle s_\alpha\mid \alpha < \dom(\overline{s}) \rangle$ the length of $\overline{s}$ (in symbols $\ell(\overline{s})$) denotes $\dom(\overline{s})$. We denote the empty sequence by $\langle\rangle$. Moreover, for sequences $\overline{s}$, $\overline{t}$, we  let $\overline{s} \tieconcat \overline{t}$ denote the natural concatenation of them (of length $\ell(\overline{s}) + \ell(\overline{t})$).
	For a set $X$, and ordinal $\alpha$ we use	$^\alpha X = \{ \overline{s}\mid \ell(\overline{s}) =  \alpha,~\ran(\overline{s}) \subseteq X\}$.

\subsection{Small cancellation theory}

	The main algebraic tool we are going to use is small cancellation theory.
	In this regard the paper is self-contained, but for more details and proofs
 the interested reader can consult \cite{lyndon1977combinatorial} and  \cite[§1]{Sh:69}.
	
	\begin{definition} \label{d21}
Given groups $H,K,L$ such that $K\cap L = H$ (as sets), in particular $H \leq K,L$, then one constructs the free amalgamation of $K$ and $L$ over $H$ as
		\[ K *_H L = F_{K \cup L} / N, \] 
		where $F_{K \cup L}$ is the free group generated by the elements of $K \cup L$,
		and
		\[ N = \ < E_K \cup E_L >^{K *_H L}, \]
		i.e., $N$ is the normal subgroup generated by $E_K \cup E_L$, where for $G \in \{K,L\}$,
		\[E_G = \{g_1g_2g_3^{-1}\mid g_1,g_2,g_3 \in G, ~g_1g_2 = g_3\}.\]
	\end{definition}
	We invoke basic results about the structure of groups of the form $ K *_H L $.
	\begin{definition}
		If $g = g^*_0 g^*_1 \ldots g^*_{n-1} \in   K *_H L$,
		where $g^*_i \in K \cup L$, then we call the sequence of $g^*_i$'s the \emph{canonical form} of the group element of $g$, if
		\begin{itemize}
			\item either $n = 1$, or
			\item $n>1$, and for each $i<n$ 
			\begin{enumerate}
				\item $g^*_i \notin H$,
				\item $i+1 <n$ $\to$ ($g^*_i \in K$ $ \iff$ $g^*_{i+1} \in L$),
			\end{enumerate}
		\end{itemize}
	\end{definition}
	
	The canonical form is unique in the following sense.
	\begin{lemma} \label{unie}
		Suppose that $ g^*_0 g^*_1 \ldots g^*_{n-1} =  g^{**}_0 g^{**}_1 \ldots g^{**}_{m-1} \in  K *_H L$ are canonical representations of the same element. Then $n = m$, and there exist
		$h_0, h_1, h_2, \ldots, h_{n} \in H$ with $h_0=h_n = \one$, and 
		\[ (\forall i <n)[g^{**}_i = h^{-1}_i g^*_i h_{i+1}].\]
	\end{lemma}

	\begin{definition} \label{dfmaln}
		Let $H \leq L$ be a pair of groups. We say that $H$ is a \emph{malnormal} subgroup of $L$, and denote it by $H \mal L$, if 
		\[ (\forall h \in H \setminus \{\one\})(\forall g \in L \setminus H)[g^{-1}hg \notin H].\]
	\end{definition}
	Note that $\mal$ is a transitive relation.
	
	\begin{definition}
		Fix $g  \in   K *_H L$ distinct from $\one$, and the canonical representation $g = g^*_0 g^*_1 \ldots g^*_{n-1}$. We say that $g^*_0 g^*_1 \ldots g^*_{n-1}$ is \emph{weakly cyclically reduced} if
		\begin{itemize}
			\item $n=1$, or 
			\item $n$ is even, or
			\item $g^*_{n-1}g^*_0 \notin H$, equivalently, $g$ has no conjugate that has a canonical representation shorter than $n-1$.
		\end{itemize}
	\end{definition}
	\begin{observation}\hfill \label{obsa}
			\begin{enumerate}
			\item 	If $g^*_0 g^*_1 \ldots g^*_{n-1}$ is a canonical representation
			of an element $g\neq \one$, $n\geq 2$, then $g$ has a conjugate $g'$ that has a canonical representation of length $m=1$, or $m=2k$ for some $k \geq 1$. Moreover, each conjugate $g''$ of $g$ has length at least $m$.
			\item \label{obsa2}	If $g^*_0 g^*_1 \ldots g^*_{n-1}$ is a canonical representation of an element $g\neq\one$, $n$ is even, and $g'$ is a weakly cyclically reduced conjugate of $g$, then $g'$ has a canonical representation in the following form:
			\[ g' = x'_i g^*_{i+1} g^*_{i+2} \ldots g^*_{n-1} g^*_{0} \ldots g^*_{i-1} x''_i, \]
			where:
			\begin{itemize}
			\item for all $g^*_i \in K$,  $x'_i, x''_i \in K$ and $K \models x''_ix'_i = g^*_i$,
			\item for all $g^*_i \in L$, $x'_i, x''_i \in L$ and $L \models x''_ix'_i = g^*_i$.
			\end{itemize}
		\end{enumerate}
	\end{observation}
	Recalling Lemma~\ref{unie} it is not difficult to see that this is a property of the element of  $K *_H L$, i.e., it does not depend on the particular choice of the canonical representation $g^*_0 g^*_1 \ldots g^*_{n-1}$.
	
	\begin{definition}
		Let $H \leq K,L$ be groups such that $L \cap K = H$, and fix $R \subseteq K *_H L$. We say that 
		$R$ is \emph{symmetrized} if for every $g \in R$:
		\begin{enumerate}
			\item $g^{-1} \in R$, and
			\item for each $g'$ that is conjugate to $g$ and  weakly cyclically reduced, $g' \in R$.
		\end{enumerate}
	\end{definition}

	\begin{definition}
		Let $X \subseteq K *_H L$, and $\chi \in (0, 1)$. We say that  $X$ satisfies $C'(\chi)$, if whenever 
		\begin{enumerate}
			\item $g^*_{n-1} g^*_{n-2} \ldots g^*_{0}, g^{**}_0 g^{**}_1 \ldots g^{**}_{m-1} \in X$,
			\item $g^*_{n-1} g^*_{n-2} \ldots g^*_1 g^*_{0} \cdot g^{**}_0 g^{**}_1 \ldots g^{**}_{m-1} \neq \one$,
			\item $\ell < n,m$, and
			\item $g^*_{\ell-1}g^*_{\ell-2} \ldots g^*_0g^{**}_0g^{**}_1 \ldots g^{**}_{\ell-1} \in H$,
		\end{enumerate}  
			then $\ell < \min(n,m) \cdot \chi$.
	\end{definition}
	
	\begin{definition}
		Let $H,K,L$ be as in Definition~\ref{d21}, and let
		$g \in K *_H L$. We say that the word $w_0w_1 \ldots w_{m-1}$ is a part of $g$, if
		\begin{enumerate}
			\item $w_0w_1 \ldots w_{m-1} \in K *_H L$ is in canonical form,
			\item for some weakly cyclically reduced conjugate $g'$ of $g$, the word $\langle w_0,w_1, \ldots, w_{m-1} \rangle$ is a subword of a canonical representation of $g'$
			(i.e., \emph{for some} canonical representation $v_0 v_1 \ldots v_{n-1}$ of $g'$ and some $k  \leq n-m$, we have $v_k = w_0$, $v_{k+1} = w_1, \ldots, v_{k+m-2} = w_{m-2}$, $v_{k+m-1} = w_{m-1}$.)
		\end{enumerate}
	\end{definition}
	We cite the following lemma, which is our key technical tool borrowed from small cancellation theory.
	\begin{lemma} \label{smacan}
		Let $H \leq K,L$ be groups, $K \cap L = H$, $k$ a positive integer, and assume that $R \subseteq K *_H L$ is symmetrized and satisfies $C'(\frac{1}{k})$.
		
		Then, letting $N= \ < R >^{K *_H L}$ be the normal subgroup generated by $R$, 
		for every weakly cyclically reduced $w \in N$, there exist $r \in R$ and a part $p$ of $r$, which is also a part of $w$, and $\ell(p) > \frac{k-3}{k} \ell(r)$. 
	\end{lemma}
	\begin{corollary} \label{smacor}
		If $H,K,L,R$ are as in Lemma~\ref{smacan}, then for the canonical projection map
		$\pi: K \ast_H L \to (K \ast_H L)/N$, it is the case that $\pi \rest K$ and $\pi \rest L$ are injective, and $ \pi``K \cap \pi``L = \pi``H$ (where $K,L$ are identified with the subgroups of $K \ast_H L$).
	\end{corollary}

\section{Finding the right amalgam}\label{pumpupsection}
The main result of this section is Lemma~\ref{shhe} below. It originates to the lemma by G.~Hesse appearing in the Appendix of \cite{Sh:69}.
The lemma will serve as a building block in the recursive construction of Section~\ref{sec4}.

		\begin{definition}
		We let $\varrho(x,y)$ denote the word $xyx^2yx^3y \ldots x^{80}y$. Note that $\ell(\varrho(x,y)) = 3320$.
	\end{definition}

In reading the statement of the next lemma, recall that $H \mal L$ means $H$ is a malnormal subgroup of $L$ (see Definition \ref{dfmaln}).
	\begin{lemma}\label{shhe}
		Let $H \leq K$, $H \mal L$ be groups, $K \cap L = H$ and suppose that 
		$$S = \{ (h_i,a_i,b_i,b_i')\mid i \in I\} \subseteq H \times (K \setminus H) \times (L\setminus H)\times (L\setminus H)$$
		satisfies the following two:
		\begin{enumerate}
			\item for every $i \in I$, $b_i$ and $b'_{i}$ are good fellows over $H$,
			\item for all $i \neq j$ in $I$, at least one of the following holds: \begin{enumerate}[label = $(\circleddash)_\alph*$, ref = $(\circleddash)_\alph*$]
						\item $a_i$ and $a_j$ are good fellows over $H$ (in $K$),
				\item \label{harm} $b_i = b_j$, $b_i' = b_j'$, and $a_i \neq a_j$, 
				\item $b_i$ and $b_j$ are good fellows over $H$,
				\item \label{harh} there are subgroups $H' \leq H$ and $K' \leq K$ such that all of the following hold:
				\begin{enumerate}[label = $(\roman*)$, ref =  $(\roman*)$]
					\item $K' \cap H = H'$,
					\item $a_i, a_j \in K' \setminus H = K' \setminus H'$,
					\item \label{gf} $b_i$ and $b_j$ are good fellows over $H'$ (in $L$),
					\item \label{gf2} $b_i$ and $b_j'$ are good fellows over $H$,
					\item $(K' \setminus H) \cdot (H\setminus K') \cdot (K' \setminus H) \subseteq (K \setminus H)$.

				\end{enumerate}  
			\end{enumerate}
		\end{enumerate}

		Then, letting $R$  be the symmetric closure of $\{h_i^{-1}\varrho(b_ia_i, b_i'a_i)\mid i \in I\}$, $M = K \ast_H L$,  $N= R^{M}$ the generated normal subgroup and $M^* = M / N$,  all of the following hold:
		\begin{enumerate}[label = $(\Alph*)$ , ref = $(\Alph*)$]
			\item $R$ satisfies the condition $C'(\frac{1}{10}$), consequently, for the canonical mapping $\pi: M \to M^*$, we have that $\pi \rest (K \cup L)$ is injective,
			\item $K \mal M^*$, and if $H \mal K$, then $L \mal M^*$,
			\item \label{gfa} if $b,b' \in L \setminus H$ are \emph{not} 
			 good fellows over $H$, $d \in K \setminus H$, then the group elements $db'$ and $dbdb$ are good fellows over $K$ in $M^*$, 
				\item \label{par} 
				if $b,b' \in L \setminus H$, $a \in K \setminus H$, then $M^* \models bab' \notin K$, $ba \notin K$ (and similarly the parallel statement with with interchanging $K$ and $L$)
			\item \label{H'} if $a$, $a' \in K$ are good fellows over $H' \leq H$ (in $K$), and $L' \leq L$ is such that $L' \cap K = L' \cap H = H'$, then they are good fellows over $L'$ in $M^*$,
			\item\label{HH}  
			if $b,b' \in L$ are good fellows over $H$, then they are good fellows over $K$ in $M^*$,
			\item \label{tfree} If $K$, $L$ are torsion-free, then so is $M^* =M/N$.
		
		\end{enumerate}
	\end{lemma}
	\begin{PROOF}{Lemma~\ref{shhe}}
		First we note that if $a \in K \setminus H$, $b,b' \in L \setminus H$, then the word $\varrho(ba, b'a)$ is  an alternating word (over the union of $K\setminus H$ and $L \setminus H$) of length $6640$.
		\begin{enumerate}[label = $(\Alph*)$ , ref = $(\Alph*)$]
			\item By Corollary \ref{smacor} it is enough to argue that $R$ satisfies $C'(\frac{1}{10})$.
To this end, fix two elements $g \neq g^* \in R$, as well as some canonical representations
			 $$g = g_0g_1\cdots g_{n-1},$$ 
			 $$g^* = g^*_0g^*_1\cdots g_{m-1}^*.$$
			 By Clause~\eqref{obsa2} of Observation~\ref{obsa}, we clearly get that $n,m \in \{6640,6641\}$.
			Let $l \in \omega$, and assume that 
			 \begin{equation} \label{felt} (k \leq l) \ \to \ (K \ast_H L \models \ g^{-1}_{k-1}g^{-1}_k\cdots g_0^{-1}g^*_0g^*_1\cdots g^*_{k-1} \in H), \end{equation}
			 so we have to show that $l \leq 664$.
			 
			 Assume on the contrary that $l > 664$.
			 We can choose $i,i^* \in I$, such that $g$ is a weakly cyclically reduced conjugate of $h_i^{-1}\varrho(b_ia_i, b_i'a_i)$
			  or of $(h_i^{-1}\varrho(b_ia_i, b_i'a_i))^{-1}$,
			  and similarly for $g^*$ and $i^*$.
			 If we fix the canonical representations
			 $$h_i^{-1}\varrho(b_ia_i, b_i'a_i) = u_0u_1 \cdots u_{6639},$$
			 where $u_j \in \{b_i,b_i',a_i,h_i^{-1}b_i\}$,
			 and similarly
			  $$h_{i^*}^{-1}\varrho(b_{i^*}a_{i^*}, b_{i^*}'a_{i^*}) = u^*_0u^*_1 \cdots u^*_{6639},$$
			 then again recalling Observation~\ref{obsa}\eqref{obsa2}, we can assume that there exist $j,j^* <6640$, $\varp,\varp^* \in \{1,-1\}$, such that
			 whenever $0 < k \leq 6640$, then $g_k = u^\varp_{j+ \varp k}$ and $g^*_k = u_{j^* + \varp^* k}^{\varp^*}$.

			 First note that the pair $b_i$, $b_i'$ are not good fellows over $H$:
			 there is a $1 \leq k \leq 10$ such that $u_{j + \varp k} \in \{b^\varp_i, (h^{-1}_ib_i)^{\varp}\}$, and $u^*_{j^* + \varp^* k} \in \{b^{\varp^*}_{i^*}, (h^{-1}_{i^*}b_{i^*})^{\varp^*}\}$, so by \eqref{felt} for some $h \in H$ we have $b^{-\varp}_ih b^{\varp^*}_{i^*} \in H$, implying that $b_i$ and $b_{i^*}$ are not good fellows over $H$. Similarly, for some $1 \leq k^\bullet \leq 2$ $u_{j + \varp k^\bullet} = a^\varp_i$, and $u^*_{j^* + \varp^* k^\bullet} = a^{\varp^*}_{i^*}$, and by the same line of reasoning $a_i$ and $a_{i^*}$ are not good fellows over $H$.
			 
			 We clearly get that 
			 \begin{enumerate}[label = $(\boxminus)$, ref =  $(\boxminus)$]
			 	\item \label{uj}	  \ref{harm}, or \ref{harh}, or $i = i^*$ holds, and in each case $b_i$, $b_{i^*}'$ are good fellows over $H$.
			 \end{enumerate}
		
			 We have to distinguish between cases depending on the values of $j,j^*,\varp,\varp^*$: if $j \neq j^*$, or $\varp \neq \varp^*$, then there exists $1 \leq k <500$ such that $u_{j + \varp k} \in \{b^\varp_i, (h_i^{-1}b_i)^\varp \}$, and $u^*_{j^* + \varp^* k} = (b'_{i^*})^{\varp^*} = (b_i')^{\varp^*}$, and for some $h \in H$ we have $b_i^{-\varp} h(b'_{i^*})^{\varp^*}$ $\in H$ (or $(h_i^{-1}b_i)^{-\varp} h(b'_{i^*})^{\varp^*}$), so $b_i$ and $b_{i^*}'$ are not good fellows over $H$, contradicting \ref{uj}.

			 Therefore we can assume that $j = j^*$ and $\varp = \varp^*$. 	 \begin{enumerate}[label = $(\boxtimes)_1$, ref =  $(\boxtimes)_1$]
			 	\item Suppose first that either \ref{harm} or $i = i^*$, and so necessarily $b_i = b_{i'}$, $b_i' = b_{i'}'$.
			 \end{enumerate}
		  But now for some $1\leq k \leq 10$ $g_k = g^*_k = b_i$, so for 
			 $$h = g^{-1}_{k-1}g^{-1}_{k-2} \cdots g^{-1}_0g^*_0g^*_1 \dots g^*_{k-1} \in H$$ we have
			 $$ g_k h g_k^{-1} \in H,$$
			 but then $H \mal L$ (together with $b_i \in L \setminus H$) implies that $h = \one$.
			
			Now if $i = i^*$, then invoking Observation~\ref{obsa}\eqref{obsa2} again (and recalling that  $g$ and $g^*$ are cyclically reduced conjugates of $h^{-1}_i \varrho(b_ia_i, b_i'a_i)$),  it is straightforward to check that $j = j^*$ and $\varp = \varp^*$ imply $g = g^*$, which is a contradiction.
			
			On the other hand, if $i \neq i'$, $a_i \neq a_{i'}$, then $g_k h g_k^{-1} = \one$ implies that
			$g_{k+1}(g_kh g_k^{-1})(g^*_{k+1})^{-1} = a_ia^{-1}_{i'} \neq \one$, and in the following step (conjugating by $b_i=b_{i'}$ again) we get a contradiction.
			
			It remains to check the case when
			 \begin{enumerate}[label = $(\boxtimes)_2$, ref =  $(\boxtimes)_2$]
				\item the pair $i$, $i^*$ satisfies condition \ref{harh}:
			\end{enumerate}
			(Recall that we can assume $j = j^*$, $\varp = \varp^*$.)
			First we claim that
			\begin{enumerate}[label = $(\intercal)$, ref =  $(\intercal)$]
				\item \label{inte}  for some $k$ with $1 \leq k < 12$ we have $g_k = u_{j + \varp k} = a^\varp_i$, and $g^*_k = u^*_{j + \varp k} = a^{\varp}_{i^*}$, and 
				$$h = g^{-1}_{k-1}g^{-1}_{k-2} \cdots g^{-1}_0g^*_0g^*_1 \cdots g^*_{k-1} \in H \setminus K' = H \setminus H'.$$
			\end{enumerate}
			
			As above for some $k^\bullet < 10$ we have $u_{j + \varp k^\bullet} = a^\varp_i$, and $u^*_{j + \varp k^\bullet} = a^{\varp}_{i^*}$, $u_{j + \varp (k^\bullet+1)} = b^\varp_i$, and $u^*_{j + \varp (k^\bullet+1)} = b^{\varp}_{i^*}$  suppose that 
			$$h = g^{-1}_{k^\bullet-1}g^{-1}_{k^\bullet-2} \cdots g^{-1}_0g^*_0g^*_1 \cdots g^*_{k^\bullet-1} \in  H'.$$
			Then $h' = a^{-\varp}_i h a^\varp_{i^*} \in K' H' K' = K'$, and by our indirect assumptions $a^{-\varp_i} h a^\varp_i \in H$, so $h'$ lie in the intersection $K' \cap H = H'$.
			Now 
			$$u_{j + \varp (k^\bullet+1)}^{-\varp} h' u^{\varp}_{j + \varp (k^\bullet+1)} = b^{-\varp}_i h' b^\varp_{i^*} \in b^{-\varp}_i H' b^{\varp}_{i^*},$$
			so by \ref{harh} \ref{gf} this product is not in $H'$, thus we can assume that some $k<12$ satisfies \ref{inte}.
			
			But then using $a_i, a_{i^*} \in K' \setminus H'$, 
			$$ g^{-1}_{k}g^{-1}_{k} \cdots g^{-1}_0g^*_0g^*_1 \cdots g^*_{k} = a_i^{-\varp}h a_{i^*}^{\varp} \in (K' \setminus H) \cdot (H \setminus H') \cdot (K' \setminus H) \subseteq K \setminus H,$$
			a contradiction.
			
			\item Fix $g,g' \in K\setminus \{\one\} \subseteq M^*$, and $z \in M^* \setminus K$, with a canonical form $z = z_0z_1 \cdots z_{m-1}$, that satisfies it does not contain any subsequence $z_{i_0}z_{i_0+1}\ldots z_{i_0+j-1}$ that is a subsequence of a canonical form of an element $r \in R$, where $j > \frac{6640}{2}+1$ (we can assume this, since otherwise we could insert the entire sequence of the inverse of this fixed canonical form of $r$). 
			Now suppose that $zgz^{-1}g' = \one$ holds in $M^*$, i.e.\
			\[ M \models zgz^{-1}g' \in N. \]
			W.l.o.g.~$z_0,z_{m-1} \in L$ (thus $m$ is odd), since otherwise we can replace $g$ with $z_{m-1}gz_{m-1}^{-1} \in K \setminus \{\one\}$, and $g'$ with $z^{-1}_0g'z_0 \in K \setminus \{\one\}$.
			This means that the product $z_0z_1 \cdots z_{m-1}gz_{m-1}^{-1} \cdots z_0^{-1}g'$ is in a  weakly cyclically reduced form, so a cyclic conjugation contains a long ($> 7/10$) subword of some canonical form of an $r \in R$. By our assumptions on $z$ (not containing more than half of a canonical representation of $r$) this has to involve either $g$ or $g'$, in fact either the word $z_{j}z_{j+1} \dots z_{m-1}gz^{-1}_{m-1} z^{-1}_{m-2} \dots z^{-1}_{j}$ or $z^{-1}_{j_*}z^{-1}_{j_*-1} \dots z^{-1}_{0}g'z_{0} z_{1} \dots z_{j_*}$ contains a long ($>2/10$ fraction) subword of a canonical form of some $r \in R$. But this is impossible since in any $r = r_0r_1 \dots r_{n-1} \in R$ ($n \in \{6640,6641\}$) at any fixed $t \in [\frac{6640}{10}, \frac{6640 \cdot 9}{10}]$ there exists $k < 250$ such that (for some $i \in I$) $r_{t-k} \in Hb_i^{\pm 1}H$, $r_{t+k} \in H(b_i')^{\pm 1}H$, and so are good fellows over $H$, while $z_k$, $z_k^{-1}$ are clearly not good fellows over $H$. 
			 \item Suppose otherwise, e.g. for some $k,k' \in K$ 
			 either $$y = (db')k(b^{-1}d^{-1}b^{-1}d^{-1})k' = \one \text{ in }M^*,$$
			 or  $$y= (db')k(dbdb)k' = \one.$$
			 Observe that after performing the cancellations in the free amalgam $M$ and writing $y = y_0y_1\ldots y_{m-1}$ as a reduced (alternating) word, in both cases (regardless of whether $k,k' \in H$) there is at most one $j$ for which $y_j \in L \setminus H$ and $y_j$ and $b$ are good fellows over $H$. Now possibly replacing $y_0y_1 \cdots y_{m-1}$ with a weakly cyclically reduced conjugate of it (if the reduced form of $y_0y_1 \cdots y_{m-1}$ is not weakly cyclically reduced) this clause remains true (and the resulting word similarly belongs to $N$ in $M$). It is not difficult to see, that there exists at least one $j'$ such that $y_{j'}$ and $b$ are not good fellows over $H$. Again, $y_0y_1 \cdots y_{m-1}$ (or a cyclical permutation of it) contains a long subword of a canonical form of some $r \in R$, but any such subword (if longer than $400$) contains at least two-two occurrences of $b_i$ and $b_i'$ (for some $i \in I$), and $b$ must be good fellow with either $b_i$ or $b_i'$ (since $b_i$, $b_i'$ are good fellows).
			 
			 \item This is the same as above. Assuming that $M^* \models bab' \in K$, then for some $a' \in K$ 
			 $$M^* \models \ bab'a' = \one,$$
			 so 
			 $$M \models \ bab'a' \in N.$$
			 Now if $a' \in K' \setminus H$, then the word $bab'a'$ is weakly cyclically reduced, so any weakly cyclically reduced conjugate to it is of length either $4$ or $5$, and clearly cannot contain a long subword of any $r \in R$.
			 
			 If $a' \in H$, then depending on whether $b'' = b'a'b \in H$, or not we have that either $b^{-1}(bab'a')b = ah \in K \setminus H$ is  weakly cyclically reduced (so $M \models bab'a' \notin b^{-1}Nb$),  or  $b^{-1}bab'a'b = ab'a'b = a b''$ (where $b'' \notin H$), which is weakly cyclically reduced, and similarly cannot lie in $N$. 
			 
			 \item Let $k,k' \in K \setminus H$ be good fellows over $H'$, and fix $l,l' \in L'$.
			 Suppose that $M^* \models klk'l' = \one$, i.e.\
			  $$M \models w= klk'l' \in N.$$
			 We can write $w$ as a reduced word. If $l \in H$, then $l \in H'$, and since $k,k'$ are good fellows over $H'$ we have $klk' \in K \setminus H$, so either $w= (klk')l'$ is a product of an element of $K \setminus H$ and $L \setminus H$ (if $l' \notin H$), or $(klk')l' \in (K \setminus H) \cdot H = K \setminus H$, we are done.
			 
			 So w.l.o.g.~$l \notin H$. (Similarly, $M^* \models k'l'kl = \one$ implies that w.l.o.g.~$l' \notin H$).
			 So any weakly cyclically reduced conjugate of $w \in M$ has length at most $5$, and contains at least $2$ entries from $K \setminus H$. But $w \in N$ implies that some weakly cyclically reduced conjugate contains a long subword of some $r \in R$, which is clearly impossible.
			 
			 \item Let $g \in M^*$, $n \in \omega$, $n>1$ be such that $g \neq 1$, $M^* \models \ g^n = \one$. (Recalling Observation \ref{obsa}) we can write $g$ as an alternating product of elements of $K\setminus H$ and $L \setminus H$
			 	$$g = g_0g_1 \cdots g_{2m-1}.$$ 
			 	W.l.o.g.~ there exists no conjugate $ygy^{-1}$ of $g$, and $g'$ with $g'(ygy^{-1})^{-1} \in N$ such that $g'$ has a shorter canonical representation than $2m$, since  we can replace $g$ with $g'$ and get a torsion element. Therefore there is no $r \in R$, $i_0 <2m$ with the sequence $g_{i_0}g_{i_0+1} \ldots g_{2m-1}g_0g_1 \ldots g_{i_0-1}$ containing a subsequence of a canonical representation of $r$ of length $j > \frac{6640}{2}+1$.
			 	
			 	Now, since 
			 	$$M \models \ (g_0g_1 \cdots g_{2m-1})^n \in N,$$
			 	there exists a cyclic conjugate of $(g_0g_1 \cdots g_{2m-1})^n$ and a subsequence $s_0s_1 \ldots s_j$ of it that is also a subsequence of a canonical form of some $s \in R$ with $j \geq \frac{7}{10} \cdot 6640$. Our assumptions above on $g_0g_1 \cdots g_{2m-1}$ easily implies
			 	$$2m \leq  \frac{6640}{2}+1,$$
			 	thus
			 	$$2m +  \frac{2}{10}\cdot 6640-1 \leq j,$$
			 	clearly $2m+330 \leq j$.
			 	This way we get that $s_\ell$ and $s_{\ell+2m}$ are \emph{not} good fellows over $H$ for each $\ell \leq 330$, but as $s$ is a cyclically reduced conjugate of $h_i^{-1}\varrho(b_ia_i,b_i'a_i)$ or of its inverse (for some $i \in I$), we get that for some $\ell \in [1, 330]$ $s_{\ell} \in Hb_i^{\pm 1}H$, $s_{\ell+2m} \in H(b'_i)^{\pm 1}H$, thus $s_\ell$ and $s_{\ell+2m}$ are good fellows over $H$, a contradiction.
			 	  \qedhere
		\end{enumerate}
	\end{PROOF}

\section{A set-theoretic interlude}\label{infcombinatorics}
\begin{definition}\label{def12}
A coloring of pairs $e:[\kappa]^2\rightarrow \theta$ is
  \emph{subadditive} if, for all $\alpha < \beta < \gamma < \kappa$, the following inequalities
  hold:
  \begin{enumerate}
    \item $e(\alpha, \gamma) \leq \max\{e(\alpha, \beta), e(\beta, \gamma)\}$;
    \item $e(\alpha, \beta) \leq \max\{e(\alpha, \gamma), e(\beta, \gamma)\}$.
  \end{enumerate}
\end{definition}
\begin{notation}\label{def41} Whenever the map $e:[\kappa]^2\rightarrow\theta$ is clear from the context,
we define for all $\gamma < \kappa$ and $i \leq \theta$, the following sets:
\begin{itemize}
\item $D^\gamma_{<i}=\{\beta < \gamma \mid e(\beta, \gamma) < i\}$, and
\item $D^\gamma_{\le i}=\{\beta < \gamma \mid e(\beta, \gamma) \le i\}$.
\end{itemize}
\end{notation}

\begin{theorem}\label{successorofregular} Suppose that $\lambda$ is an infinite regular cardinal.
Then there exist two maps $c:[\lambda^+]^2\rightarrow\lambda^+\times\lambda^+$ and $e:[\lambda^+]^2\rightarrow\lambda$ 
such that:
	\begin{itemize}
	\item $e$ is subadditive;
	\item for every $A\in[\lambda^+]^{\lambda^+}$, there exists a club $D\s\lambda^+$ such that for every $\delta\in D$,
	for every $\beta\in\lambda^+\setminus\delta$, for every $(\xi_0,\xi_1)\in\delta\times\delta$, 
	for every $i<\lambda$, there are cofinally many $\alpha<\delta$ such that $\alpha\in A$, $c(\alpha,\beta)=(\xi_0,\xi_1)$ and $e(\alpha,\beta)>i$.
	\end{itemize}
\end{theorem}
\begin{proof} Let $e$ be the function $\rho:[\lambda^+]^2\rightarrow\lambda$ defined in \cite[\S9.1]{todorcevic_book}.
  By \cite[Lemma~9.1.1]{todorcevic_book}, $e$ is subadditive.
  By \cite[Lemma~9.1.2]{todorcevic_book}, $e$ is also \emph{locally small},
  i.e., $|D^\gamma_{\le i}|<\lambda$ for all $\gamma<\lambda^+$ and $i<\lambda$.
  
  Next, by \cite{paper14}, we may fix a coloring $d:[\lambda^+]^2\rightarrow\lambda^+$ witnessing $\lambda^+\nrightarrow[\lambda^+;\lambda^+]^2_{\lambda^+}$.
  By \cite[Lemma~3.16]{paper53}, this means that for every $A\in[\lambda^+]^{\lambda^+}$, there exists an $\epsilon<\lambda^+$ such that, for all $\beta\in\lambda^+\setminus\epsilon$ and $\tau<\epsilon$,
there exists $\alpha\in A\cap\epsilon$ such that $d(\alpha,\beta)=\tau$.
Fix a bijection $\pi:\lambda^+\leftrightarrow\lambda^+\times\lambda^+$,
and then let $c$ be the composition $\pi\circ d$.

We now verify that $c$ and $e$ are as sought.

\begin{claim} Let $A\in[\lambda^+]^{\lambda^+}$. Then there exists a club $D\s\lambda^+$ such that for every $\delta\in D$,
	for every $\beta\in\lambda^+\setminus\delta$,
	for every $(\xi_0,\xi_1)\in\delta\times\delta$, 
	for every $i<\lambda$, there are cofinally many $\alpha<\delta$ such that $\alpha\in A$, $c(\alpha,\beta)=(\xi_0,\xi_1)$ and $e(\alpha,\beta)>i$.
\end{claim}
\begin{proof} Let $\langle M_\gamma\mid \gamma<\lambda^+\rangle$ be a sequence of elementary submodels of $H_{\lambda^{++}}$,
each of size $\lambda$, such that $\{A,d,\pi\}\in M_0$, such that $M_\gamma\in M_{\gamma+1}$ for every $\gamma<\lambda^+$,
and such that $M_\delta=\bigcup_{\gamma<\delta}M_\gamma$ for every limit nonzero $\delta<\lambda^+$.
It follows that $C=\{\gamma<\lambda^+\mid M_\gamma\cap\lambda^+=\gamma\}$ is a a club in $\lambda^+$.

We claim that the following club is as sought:
$$D=\{\delta<\lambda^+\mid \otp(C\cap\delta)=\lambda^\delta\}.$$
To this end, let $\delta\in D$, $\beta\in\lambda^+\setminus\delta$,
$(\xi_0,\xi_1)\in\delta\times\delta$, $i<\lambda$, and $\eta<\delta$. We shall find an $\alpha\in A\cap\delta$ above $\eta$ such that $c(\alpha,\beta)=(\xi_0,\xi_1)$ and $e(\alpha,\beta)>i$.

As $\delta$ in particular belongs to $C$, $\pi[\delta]=\delta\times\delta$,
so we may fix some $\tau<\delta$ such that $\pi(\tau)=(\xi_0,\xi_1)$.
For every $\gamma\in C\setminus\tau$, the set $A_\gamma=A\setminus\gamma$ is in $[\lambda^+]^{\lambda^+}\cap M_{\gamma+1}$,
and hence there exists $\epsilon\in\lambda^+\cap M_{\gamma+1}$ such that, for all $\beta'\in\lambda^+\setminus\epsilon$ and $\tau'<\epsilon$,
there exists $\alpha'\in A_\gamma\cap\epsilon$ such that $d(\alpha',\beta')=\tau'$.
In particular, we may pick $\alpha_\gamma\in A\cap M_{\gamma+1}\setminus\gamma$ such that $d(\alpha_\gamma,\beta)=\tau$.
It follows that $\gamma\mapsto \alpha_\gamma$ is a strictly increasing function from $C\cap\delta$ to $A\cap\delta$.
As $\delta\in D$, we infer that $A'=\{ \alpha\in A\cap\delta\mid \eta<\alpha\ \&\ d(\alpha,\beta)=\tau\}$ has size $\lambda$.
As $e$ is locally small, we may now pick $\alpha\in A'\setminus D^\beta_{\le i}$. Then $\alpha\in A\cap\delta$ above $\eta$, $e(\alpha,\beta)>i$
and $c(\alpha,\beta)=\pi(d(\alpha,\beta))=\pi(\tau)=(\xi_0,\xi_1)$, as sought.
\end{proof}
This completes the proof.
\end{proof}

\begin{remark} The preceding result does not generalize to the case when $\lambda$ is a singular cardinal.
Indeed, it follows from \cite[Lemma~3.38]{paper36} that if $\lambda$ is the singular limit of strongly compact cardinals,
then for every cardinal $\theta\le\lambda$, for every subadditive map $c:[\lambda^+]^2\rightarrow\theta$,
there must exist an $A\in[\lambda^+]^{\lambda^+}$ such that $\sup\{ c(\alpha,\beta)\mid \alpha<\beta\text{ in }A\}<\theta$.
\end{remark}

\begin{theorem}\label{generalcase} Suppose that $\theta<\kappa$ are infinite regular cardinals, and $\square(\kappa)$ holds.

Then there exist two maps $c:[\kappa]^2\rightarrow\kappa$ and $e:[\kappa]^2\rightarrow\theta$ 
such that:
	\begin{itemize}
	\item $e$ is subadditive;
	\item for every $A\in[\kappa]^\kappa$, there exists a club $D\s\kappa$ such that for every $\delta\in D$,
	for every $\beta\in\kappa\setminus\delta$, for every $(\xi_0,\xi_1)\in\delta\times \delta$, 
	for every $i<\theta$, there are cofinally many $\alpha<\delta$ such that $\alpha\in A$, $c(\alpha,\beta)=(\xi_0,\xi_1)$ and $e(\alpha,\beta)>i$.
	\end{itemize}
\end{theorem}
\begin{proof} By Theorem~\ref{successorofregular}, we may assume that $\theta^+<\kappa$. 
It thus follows from \cite[Theorem~A$'$]{paper52} that we may fix a coloring $d:[\kappa]^2\rightarrow\kappa$
witnessing $\pr_1(\kappa,\kappa,\kappa,\theta^+)$.
By \cite[Lemma~4.2]{paper45}, this means that for every $\tau<\kappa$,
for every pairwise disjoint subfamily $\mathcal A\s[\kappa]^\theta$  of size $\kappa$,
there exists $\epsilon<\kappa$ such that, for every $b\in[\kappa\setminus\epsilon]^\theta$, for some $a\in\mathcal A\cap\mathcal P(\epsilon)$,
$d[a\times b]=\{\tau\}$.
Fix a bijection $\pi:\kappa\leftrightarrow\kappa\times\kappa$,
and then let $c$ be the composition $\pi\circ d$.

Next, as a second application of $\square(\kappa)$, by \cite[Theorem~A]{paper36}, we may pick a subadditive witness $e:[\kappa]^2\rightarrow\theta$ to $\U(\kappa,2,\theta,2)$.
The latter means that $e``[S]^2$ is cofinal in $\theta$ for every $S\in[\kappa]^\kappa$.
\begin{claim}\label{claim341} Let $A\in[\kappa]^\kappa$.
There are club many $\gamma<\kappa$ such that $e[(A\cap\gamma)\times\{\gamma\}]$ is cofinal in $\theta$.
\end{claim}
\begin{proof} Towards a contradiction, suppose that $A\in[\kappa]^\kappa$ is a counterexample,
so that  $\{\gamma<\kappa\mid \sup(e[(A\cap\gamma)\times\{\gamma\}])<\theta\}$ is stationary.
Fix $\tau_0<\theta$ such that $S_0=\{\gamma\in\acc^+(A)\mid \sup(e[(A\cap\gamma)\times\{\gamma\}])=\tau_0\}$ is stationary.
For every $\gamma\in S_0$, let $\alpha_\gamma=\min(A\setminus(\gamma+1))$.
Then find $\tau_1<\theta$ such that $S_1=\{\gamma\in S_0\mid e(\gamma,\alpha_\gamma)=\tau_1\}$ is stationary.
As $e$ is subadditive, it follows that for every pair $\gamma<\gamma'$ of ordinals in $S_1$, $e(\gamma,\gamma')\le\max\{e(\gamma,\alpha_\gamma),e(\alpha_\gamma,\gamma')\}=\max\{\tau_1,\tau_0\}$.
So, $e``[S_1]^2$ is bounded in $\theta$,
contradicting the fact that $e$ witness $\U(\kappa,2,\theta,2)$.
\end{proof}

We now verify that $c$ and $e$ are as sought.

\begin{claim} Let $A\in[\kappa]^{\kappa}$. Then there exists a club $D\s\kappa$ such that for every $\delta\in D$,
	for every $\beta\in\kappa\setminus\delta$,
	for every $(\xi_0,\xi_1)\in\delta\times\delta$, 
	for every $i<\theta$, there are cofinally many $\alpha<\delta$ such that $\alpha\in A$, $c(\alpha,\beta)=(\xi_0,\xi_1)$ and $e(\alpha,\beta)>i$.
\end{claim}
\begin{proof} For every $\eta<\kappa$, let $C_\eta$ be the club given by Claim~\ref{claim341} with respect to the set $A\setminus\eta$.
In particular, for every $\gamma\in C_\eta$, 
$e[(A\cap(\eta,\gamma))\times\{\gamma\}]$ is cofinal in $\theta$.
Consider the club $C=\diagonal_{\eta<\kappa}C_\eta$.
Let $\Gamma$ denote the collection of all successive elements of $C$.
It follows that for every $\gamma\in\Gamma$, we may pick $a_\gamma\in[A\cap\gamma]^\theta$ such that:
\begin{itemize}
\item for every $\alpha\in a_\gamma$, $\sup(C\cap\gamma)<\alpha$;
\item $e[a_\gamma\times\{\gamma\}]$ is cofinal in $\theta$.
\end{itemize}

Let $\langle M_\gamma\mid \gamma<\kappa\rangle$ be a sequence of elementary submodels of $H_{\kappa^+}$,
each of size less than $\kappa$, such that $\{\langle a_\gamma\mid \gamma\in\Gamma\rangle,d,\pi\}\in M_0$, such that $M_\gamma\in M_{\gamma+1}$ for every $\gamma<\kappa$,
and such that $M_\delta=\bigcup_{\gamma<\delta}M_\gamma$ for every limit nonzero $\delta<\kappa$.
We claim that the following club is as sought:
$$D=\{\delta<\kappa\mid M_\delta\cap\kappa=\delta\}.$$
To this end, let $\delta\in D$, $\beta\in\kappa\setminus\delta$,
$(\xi_0,\xi_1)\in\delta\times\delta$, $i<\theta$, and $\eta<\delta$. We shall find an $\alpha\in A\cap\delta$ above $\eta$ such that $c(\alpha,\beta)=(\xi_0,\xi_1)$ and $e(\alpha,\beta)>i$.

As $\delta\in D$, it is the case that $\pi[\delta]=\delta\times\delta$,
so we may fix some $\tau<\delta$ such that $\pi(\tau)=(\xi_0,\xi_1)$.
Set $\eta^*=\min(\Gamma\setminus\max\{\tau,\eta\})$.
As $\Gamma\in M_\delta\cap[\kappa]^\kappa$, 
the collection $\mathcal A=\{ a_\gamma\mid \gamma\in\Gamma, \gamma>\eta^*\}$ is a pairwise disjoint subfamily of $[\kappa]^\theta$ of size $\kappa$,
lying in $M_\delta$. It thus follows that there exists $\epsilon\in M_\delta\cap\kappa$ such that, for every $b\in[\kappa\setminus\epsilon]^\theta$, for some $a\in\mathcal A\cap\mathcal P(\epsilon)$,
$d[a\times b]=\{\tau\}$. As $\epsilon<\delta\le\beta$, we may now pick $\gamma\in\Gamma$ with $\gamma>\eta^*$
such that $a_\gamma\in \mathcal P(\delta)$ and $d[a_\gamma\times\{\beta\}]=\{\tau\}$.
Set $j=e(\gamma,\beta)$. Recalling that $e[a_\gamma\times\{\gamma\}]$ is cofinal in $\theta$,
we may now pick $\alpha\in a_\gamma$ with $e(\alpha,\gamma)>\max\{i,j\}$.
Note that $\eta\le\eta^*\le\sup(C\cap\gamma)<\alpha<\gamma<\delta$.
In addition, since $e$ is subadditive, $$\max\{i,j\}<e(\alpha,\gamma)\le\max\{e(\alpha,\beta),e(\gamma,\beta)\}\le\{e(\alpha,\beta),j\},$$
and hence $i<e(\alpha,\beta)$.
Evidently, $c(\alpha,\beta)=\pi(d(\alpha,\beta))=\pi(\tau)=(\xi_0,\xi_1)$.
Altogether, $\alpha$ is an element of $A\cap\delta$ above $\eta$ satisfying the required properties.
\end{proof}
This completes the proof.
	\end{proof}

\section{A construction of a Shelah group}\label{sec4}

This section is devoted to proving the core result of this paper.
The assumptions of the upcoming theorem are motivated by the results of the previous section.

\begin{theorem}\label{thm41} Suppose:
	\begin{itemize}
	\item $\theta<\kappa$ is a pair of infinite regular cardinals;
	\item $c_0:[\kappa]^2\rightarrow\kappa$ and $c_1:[\kappa]^2\rightarrow\kappa$ are two colorings;
	\item $e:[\kappa]^2\rightarrow\theta$ is a subadditive coloring;
	\item for every $A\in[\kappa]^\kappa$, there exists a club $D\s\kappa$ such that for every $\delta\in D$,
	for every $\beta\in A\setminus\delta$,
	for every $(\xi_0,\xi_1)\in\delta\times\delta$, 
	for every $i<\theta$, there are cofinally many $\alpha\in A\cap\delta$ such that 
	$$c_0(\alpha,\beta)=\xi_0\ \&\ c_1(\alpha,\beta)=\xi_1 \ \&\ e(\alpha,\beta)>i.$$
	\end{itemize}
	
Then there exists a torsion-free Shelah group $G$ of size $\kappa$.
\end{theorem}
Before embarking on the proof, we make a few promises.
\subsection{Promises}
\begin{enumerate}[label = $(p)_\arabic*$, ref = $(p)_\arabic*$]
\item We shall recursively construct distinguished group elements $\langle x_\alpha\mid \alpha<\kappa\rangle$ generating the whole group $G$;
\item For every set $A\s\kappa$, $G_A$ will denote the group generated by $\{ x_\alpha\mid \alpha\in A\}$, so that $G=G_\kappa$;
\item\label{underlying} For every $\gamma\le\kappa$, the underlying set of $G_\gamma$ will be an initial segment of $\kappa$;
\item\label{maine2} For all $\gamma<\kappa$ and $i<\theta$, $G_{D^\gamma_{< i} \cup \{\gamma\}} \cap G_{\gamma} = G_{D^\gamma_{< i}}$;\footnote{Recall Notation~\ref{def41}.}
\item\label{maine} For all $\gamma<\kappa$ and $i<\theta$, $G_{D^\gamma_{< i} \cup \{\gamma\}} \cap G_{D^\gamma_{\leq i}} = G_{D^\gamma_{< i}}$;
\item \label{indhyp2} For all $\gamma< \kappa$, $i < \theta$, $ G_{D^\gamma_{< i}} \mal G_{D^\gamma_{< i} \cup \{\gamma\}}$;
\item For all $\gamma<\kappa$ and nonzero $i<\theta$,
 $G_{D^\gamma_{\le i}}$ is the group $M^*$ given by Lemma~\ref{shhe} when invoked with the groups
				\begin{itemize}
					\item $H = G_{D^\gamma_{< i}}$,
					\item  $ K = G_{D^\gamma_{\le i}}$,
					\item  $L = G_{D^\gamma_{< i} \cup \{\gamma\}}$,
				\end{itemize} 
and an appropriate system $S$.
\item\label{tmatrix} We shall also recursively construct a matrix of sets $\langle T^\gamma_{< i, \beta}\mid \beta\le\gamma<\kappa,~ i<\theta\rangle$ with the property that:
	\begin{enumerate}[label = $(t)_\arabic*$, ref = $(t)_\arabic*$]
		\item\label{fc} For all $\beta\le\gamma<\kappa$ and $ i<j<\theta$, $T^\gamma_{< i,\beta } \subseteq T^\gamma_{< j,\beta}\s G_{D^\gamma_{< i} \cup \{\gamma\}}$, and 
		\item\label{sc} For all $\alpha < \beta \leq \gamma<\kappa$ and $i < \theta$, $T^\gamma_{< i,\alpha } \supseteq T^\gamma_{< i,\beta}$.
	\end{enumerate}
\end{enumerate}

At the outset, we also agree on the following piece of notation.
\begin{notation}For all $\gamma<\kappa$ and $g\in G_\gamma$, let 
		$$i^\gamma_g = \min \{i < \theta\mid g \in  G_{D^{\gamma}_{\le i}}\}.$$
We shall also record the first appearance of an element $g\in G_\kappa$, by letting
		$$ \tau_g = \min \{\beta< \kappa\mid g \in G_{\beta+1}\}.$$
As $g\in G_{\tau_g\cup\{\tau_g\}}$ and $\tau_g=\bigcup_{i<\theta}D^{\tau_g}_{\le i}$, it also makes sense to define 
		$$ i_g = \min \{i < \theta\mid g \in  G_{D^{\tau_g}_{\le i} \cup \{\tau_g\}}\}.$$
\end{notation}
Note that since $D^\gamma_{< 0}=\emptyset$, the group $\langle x_\gamma\rangle$ generated by $x_\gamma$ will have a trivial intersection with $G_\gamma$.
Another observation worth recording is the following.

\begin{observation} Let $\gamma<\kappa$ and $i< \theta$. Then \ref{maine} implies \ref{maine2}.
\end{observation}
\begin{proof} Assuming \ref{maine}, we prove by induction on $j<\theta$ that 
		$$G_{D^\gamma_{< i} \cup \{\gamma\}} \cap G_{D^\gamma_{< j}} = G_{D^\gamma_{< i}}.$$
This is clear for $j = i+1$, and by continuity for limit $j$. Now 
		$G_{D^\gamma_{< j} \cup \{\gamma\}} \cap G_{D^\gamma_{\le j}} = G_{D^\gamma_{< j}}$  by our assumptions, so since $G_{D^\gamma_{< i} \cup \{\gamma\}}  \s G_{D^\gamma_{< j} \cup \{\gamma\}}$ trivially holds,
		$$G_{D^\gamma_{< i} \cup \{\gamma\}} \cap G_{D^\gamma_{\le j}} \s G_{D^\gamma_{< j}} \cap G_{D^\gamma_{< i} \cup \{\gamma\}} = G_{D^\gamma_{< i}},$$
as sought.\end{proof}	

To state another consequence of our promises, we agree to say that a set $F\s\kappa$ is \emph{closed} if for every $\alpha \in F$,
there exists some $i\le\theta$ such that $F \cap \alpha = D^\alpha_{< i}$.

		\begin{lemma}\label{clcor} Suppose that $F,F'$ are closed subsets of $\kappa$. 
			\begin{enumerate}
				\item \label{egy} For every $g \in G_F$,
				$D^{\tau_g}_{\leq i_g}\cup\{\tau_g\}\subseteq F$; 
				\item \label{ke} $G_{F} \cap G_{F'} = G_{F \cap F'}$.
			\end{enumerate}
		\end{lemma}
		\begin{proof}
			
			\eqref{egy} 
						This follows from the following claim, using  $\gamma=\tau_g+1$.
			\begin{claim}\label{scl}
			 If $g \in G_{F}$, then for every $\gamma <\kappa$, if $g \in G_\gamma$, then $g \in G_{F \cap \gamma}$.
			\end{claim}
			\begin{proof}
		Observe that it is enough to argue that
		whenever $\alpha_0,\alpha_1, \dots, \alpha_{n-1} \in F$ and $\alpha_*= \max \{\alpha_j\mid j <n\}$, we have
		$$(x_{\alpha_0}x_{\alpha_1}  \cdots x_{\alpha_{n-1}} \in G_{\alpha_*}) \ \to \ (x_{\alpha_0}x_{\alpha_1}  \cdots x_{\alpha_{n-1}} \in G_{F \cap \alpha_*}).$$
Fix $i < \theta$ such that $F \cap \alpha_* = D^{\alpha_*}_{< i}$.
Then $g \in G_{D^{\alpha_*}_{< i} \cup \{\alpha_*\}} \cap G_{\alpha_*}$, and hence, by \ref{maine2},
			$g \in G_{D^{\alpha_*}_{< i}} = G_{F \cap \alpha_*}$.
			\end{proof}
			
\eqref{ke} This follows from Clause~\eqref{egy}, as follows. If $g \in G_{F} \cap G_{F'}$, then
			$\{\tau_g\} \cup D^{\tau_g}_{\leq i_g} \subseteq F \cap F'$, so 
			$\{\tau_g\} \cup D^{\tau_g}_{\leq i_g} \subseteq F \cap F'$, but $g \in \{\tau_g\} \cup D^{\tau_g}_{\leq i_g}$ by the definition of $i_g$.
\end{proof}

\subsection{The recursive construction} \label{reccons}
We are now ready to start the recursive process.
We start by letting $\langle x_\alpha\mid\alpha<\theta\rangle$ be a sequence of independent group elements, so that it generates a free group with $\theta$ many generators.
We assume $G_\theta$ has underlying set $\theta$.
Hereafter, we shall not worry about \ref{underlying}, since it is clear it may be secured.
Next, suppose that $\gamma<\kappa$ is such that $G_\gamma$ has already been defined and satisfies all of our promises.
		We construct $G_{\gamma+1}$ by the following procedure. We let $x_\gamma=\min(\kappa\setminus G_\gamma)$,
		and now we need to describe the group relationship between $x_\gamma$ and the elements of $G_\gamma$.
		We will define $\langle G_{D^\gamma_{< i} \cup \{ \gamma\}}\mid i < \theta \rangle$ by recursion on $i<\theta$, in such a way that:
\begin{equation}\label{workhyp} G_{D^\gamma_{< i}} \mal G_{D^\gamma_{< i} \cup \{ \gamma\}},\quad(i<\theta)
\end{equation}

	Here we go.
			As $D^\gamma_{< 0} = \emptyset$, we mean $G_{D^\gamma_{<0}} = \{\one\}$, and we let $G_{D^\gamma_{<0} \cup \{\gamma\}} = G_{\{\gamma\}}$ be the infinite group $\mathbb Z$ generated by $x_\gamma$. Note that $G_{D^\gamma_{<0}} \mal G_{D^\gamma_{<0} \cup \{\gamma\}}$ vacuously holds. Moving on, suppose that $i<\theta$ is such that $G_{D^\gamma_{<i}\cup\{\gamma\}}$ has already been defined.
For every $\beta \leq \gamma$,  
define an equivalence relation $E^\gamma_{< i,\beta}$ over $G_{D^\gamma_{< i} \cup \{\gamma\}}$
 via:
			$$g\mathrel{E^\gamma_{< i,\beta}}h\text{ iff } g \in G_{D^\gamma_{< i} \cap \beta}h^{\pm 1} G_{D^\gamma_{< i} \cap \beta}.$$

\begin{lemma}\label{rest} For every $j \leq i$,
$ E^\gamma_{< i, \beta} \rest G_{D^\gamma_{< j} \cup \{\gamma\}} = E^\gamma_{< j, \beta}$.
\end{lemma}
\begin{proof}
By a straightforward induction on $j$, we argue that
$$ E^\gamma_{< j+1, \beta} \rest G_{D^\gamma_{< j} \cup \{\gamma\}} = E^\gamma_{< j, \beta}.$$
By continuity, this suffices.
Recalling that $G_{D^\gamma_{\le j}}$ was given by Lemma~\ref{shhe} we let $K' = G_{D^\gamma_{\le j} \cap \beta}$, so 
		$$K' \cap L = G_{D^\gamma_{\le j} \cap \beta} \cap G_{D^\gamma_{<j} \cup \{\gamma\}} =   G_{D^\gamma_{<j} \cap \beta},$$
		where the second equality is by Lemma~\ref{clcor}\eqref{ke}. 
		Now for all $g,g' \in G_{D^\gamma_{< j} \cup \{\gamma\}}$
		constituting a pair of good fellows over $K' \cap L$ must remain good fellow over 
		$$K' =  G_{D^\gamma_{\le j} \cap \beta} = G_{D^\gamma_{< j+1} \cap \beta},$$ by Clause~\ref{HH} of Lemma~\ref{shhe}.
\end{proof}

In regards to \ref{tmatrix}, we shall want $T^\gamma_{< i, \beta} \subseteq G_{D^\gamma_{< i} \cup \{\gamma\}}$ to be a transversal for the equivalence classes of $E^\gamma_{< i, \beta}$.
\begin{lemma}\label{Trest} The system of transversal sets $\langle T^\gamma_{<i,\beta}\mid \beta\le\gamma\rangle$ can be chosen as promised, satisfying \ref{fc}, \ref{sc}.
\end{lemma}
\begin{proof} Fix a well ordering $\prec_{\gamma,i}$ on $G_{D^\gamma_{< i} \cup \{\gamma\}} \setminus G_\gamma$  so that 
$$j<i \ \to \ \prec_{\gamma,i} \rest (G_{D^\gamma_{< j} \cup \{\gamma\}} \setminus G_\gamma) \ = \  \prec_{\gamma,j},$$ and 
			$$ (i_g < i_h)\implies(g \prec_\gamma h),$$ 
			and then let 
			$$T^\gamma_{< i, \beta}= \{g \in G_{D^\gamma_{<i} \cup \{\gamma\}} \mid g = \min_{\prec_{\gamma,i}}([g]_{E^\gamma_{< i, \beta}}) \}.$$
			Observe that Lemma~\ref{rest} implies \ref{fc} and \ref{sc} follow from $E^\gamma_{< i, \alpha} \subseteq E^\gamma_{< i, \beta}$.
\end{proof}		
				
		Note that
		\begin{enumerate}[label = $(\star)$, ref = $(\star)$]
			\item \label{fontt2} If $g \in G_{\gamma+1} \setminus G_\gamma$, $\beta \leq \gamma$, and $i_g \leq i$, then 
			for the unique $t \in  T^\gamma_{< i+1,\beta}$  there exist $y_0,y_1 \in G_{D^\gamma_{\le i} \cap \beta}$, $\varp \in \{-1,1\}$, such that $g = y_0t^\varp y_1$.
			Now if $\beta$ is limit, then $y_0,y_1 \in G_{D^\gamma_{\le i} \cap \alpha}$ for some $\alpha < \beta$, so $t$ and $g$ are $E^\gamma_{< i+1,\alpha}$-equivalent too (and $t \in  T^\gamma_{< i+1,\alpha}$, since $T^\gamma_{< i +1,\alpha } \supseteq T^\gamma_{< i+1,\beta}$  by \ref{sc}).
		\end{enumerate}

Pick a function $ \vec q : \kappa\to {}^3 \kappa \times \{1,-1\}$ that is surjective, i.e., for all $\zeta_0, \zeta_1, \zeta_2 \in \kappa$ and $ \varp \in \{1,-1\}$ there exists  $\xi<\kappa$ such that 
			$$(q_0(\xi), q_{1}(\xi), q_{2}(\xi), q_{3}(\xi)) = (\zeta_0,\zeta_1,\zeta_2, \varp),$$
			which we fix throughout the recursive construction over $\gamma$ and $i$.

\begin{definition}\label{constJ+} 		Let
$J^+ = \{ (l,k) \in (\bigcup_{\beta < \gamma} T^\gamma_{< i,\beta}) \times (K \setminus H)\mid \tau_k \in D^\gamma_{\leq i} \setminus D^\gamma_{<i} \ \wedge \ l \in T^\gamma_{< i,\tau_k} \}.$
\end{definition}
We shall define the tuples  $(b_j,b_{j}', a_j, h_j)$ for each $j \in J^+$, and then we will set $S = \{ (b_j,b'_{j}, a_j, h_j)\mid j \in J \}$ for an appropriate subset $J \subseteq J^+$.

\begin{definition}\label{relde} Let $J^*$ be the collection of all $(l,k) \in (\bigcup_{\beta < \gamma} T^\gamma_{< i,\beta}) \times (K \setminus H)$ 
such that $q_{\ell}(c_1(\tau_k,\gamma)) \in G_\gamma \subseteq \kappa$ for each $\ell <3$.
	
	 Now for $\sigma=(l,k) \in J^*$ define the extended tuple 
			$$ (b_{\sigma},b'_{\sigma}, a_{\sigma}, h_{\sigma}, \alpha_{\sigma}, d_{\sigma}, y_{\sigma,0},y_{\sigma,1}, \varp_{\sigma}, K_{\sigma})$$
			as follows:
			\begin{itemize}
				\item $a_{\sigma} = k$,
				\item $\alpha_{\sigma} = \tau_k$,
				\item $d_{\sigma} = q_{2}(c_1(\alpha_{\sigma},\gamma))$,
				\item $y_{\sigma,\ell} = q_{\ell}(c_1(\alpha_{\sigma},\gamma))$ for $\ell = 0,1$,
				\item $\varp_{\sigma} = q_{3}(c_1(\alpha_{\sigma},\gamma))$,
				\item $b_{\sigma} = y_{\sigma,0} \cdot l^{\varp_\sigma} \cdot y_{\sigma,1} d_\sigma$,
				\item $b_\sigma' = b_\sigma \cdot b_\sigma$,
				\item $h_\sigma = c_0(\alpha_\sigma, \gamma)$,
				\item $K'_\sigma = G_{D^\gamma_{\le i}}  \cap G_{\alpha_\sigma} (=G_{D^\gamma_{\le i} \cap \alpha_\sigma})$.
			\end{itemize}
		\item \label{constJ}
			Then we  let
		$$ \begin{array}{rclll} \sigma = (l,k) \in J & \iff &  \boxdot_1  & \tau_{d_\sigma}, \tau_{y_{\sigma,0}}, \tau_{y_{\sigma,1}} < \alpha_\sigma & (\text{i.e. }d_\sigma, y_{\sigma,0}, y_{\sigma,1} \in G_{\alpha_\sigma}), \\
			& & \boxdot_2 &  \max\{i_l,i^\gamma_{y_{\sigma,0}}, i^\gamma_{y_{\sigma,1}}\} < i^\gamma_{d_\sigma} &(\text{i.e.  for some } j  < i \\
			& & & & y_{\sigma,0},y_{\sigma,1},t \in G_{D^\gamma_{<j}}, \  d_\sigma \notin G_{D^\gamma_{<j}}), \\
			& & \boxdot_3 & h_\sigma \in G_{D^\gamma_{< i}}.
			\end{array}
			$$
	\end{definition}

		\begin{lemma} Let $\sigma = (l,k) \in J$. \label{ellem}
				\begin{enumerate}[label = $(\alph*)$, ref =  $(\alph*)$]
					\item \label{cl1} $b_\sigma$ and $b_\sigma'$ are good fellows over $H = G_{D^\gamma_{< i}}$, 
					\item \label{cl2} 	whenever $\sigma' = (l,k') \in J$ with $\sigma\neq\sigma'$, at least one of the following holds
					\begin{itemize}
						\item $a_\sigma = k$, and $a_{\sigma'} = k'$ are good fellows over $H$ (in $G_{D^\gamma_{\le i}}$),
						\item or $b_\sigma$ and $b_{\sigma'}$ are good fellows over $H = G_{D^\gamma_{<i}}$,
						\item or $b_\sigma = b_{\sigma'}$, $b_\sigma' = b_{\sigma'}'$, $a_\sigma \neq a_{\sigma'}$, 
						\item or $\alpha_\sigma = \alpha_{\sigma'}$, (so $K'_{\sigma} = K'_{\sigma'}$), $b_\sigma$ $b_{\sigma'}$ are good fellows over $K'_\sigma \cap H$, $b_{\sigma}$, $b'_{\sigma'}$ are good fellows over $H$, and
						\begin{equation} \label{K'rel}  K \models \ (K_\sigma' \setminus H) \cdot (H \setminus K_\sigma') \cdot (K_\sigma' \setminus H) \subseteq K \setminus H. \end{equation}
					\end{itemize}
				\end{enumerate}
		\end{lemma}
		\begin{PROOF}{Lemma~\ref{ellem}}
		Clause~\ref{cl1} holds by our inductive assumptions on the construction: if $i^* < i$ is such that
		$$\max\{i_l,i^\gamma_{y_{\sigma,0}}, i^\gamma_{y_{\sigma,1}}\} \leq i^* < i^\gamma_{d_\sigma} < i,$$
		then $l \in G_{D^\gamma_{\le i^*} \cup \{ \gamma \}}$,
		$$y_{j,0}, y_{\sigma,1} \in G_{D^\gamma_{\le i^*}} \subseteq G_{D^\gamma_{\le i^*} \cup \{ \gamma \}},$$
		and $d_\sigma \in G_{D^\gamma_{\le i^\gamma_{d_\sigma}}} \setminus  G_{D^\gamma_{<i^\gamma_{d_\sigma}}}$, and as $G_{D^\gamma_{\le i^\gamma_{d_\sigma} \cup \{\gamma\}}}$ has been obtained by applying Lemma~\ref{shhe}, we have that $b_\sigma$ and $b_\sigma'$ are good fellows over $G_{D^\gamma_{\leq i^\gamma_{d_\sigma}}}$. 
		Since $$E^\gamma_{< i, \gamma} \rest G_{D^\gamma_{< i^\gamma_{d_\sigma}+1} \cup \{\gamma\}} = E^\gamma_{< i^\gamma_{d_\sigma}+1, \beta}$$
		by Lemma \ref{rest} $\neg (b_\sigma E^\gamma_{< i^\gamma_{d_\sigma}+1, \gamma} b_\sigma')$ indeed implies $\neg (b_\sigma E^\gamma_{< i, \gamma} b_\sigma')$

	 Secondly, for Clause~\ref{cl2}, suppose that $\sigma = (l,k)$ and $\sigma' = (l',k')$ are two distinct elements of $J$.
The next claim takes care of the case $\tau_k \neq \tau_{k'}$.
		\begin{claim} \label{goodfella}
		  Suppose $k,k' \in G_{D^\gamma_{\leq i}}$ are such that $\tau_k <\tau_{k'} <\gamma$ and $\tau_k, \tau_{k'} \in D^\gamma_{\leq i} \setminus D^\gamma_{<i}$.
		  Then $k$ and $k'$ are good fellows over $G_{D^\gamma_{< i}}$.
		\end{claim}
		\begin{PROOF}{Claim~\ref{goodfella}}
			Note that $g \in G_{D^\gamma_{\leq i}}$, $\tau_{g} \in D^\gamma_{\leq i} \setminus D^\gamma_{< i}$ implies that $g \notin G_{D^\gamma_{< i}}$ (this is by Lemma~\ref{clcor} applying to $F = D^\gamma_{< i}$).
In addition, note that $k$, $k'$ are good fellows over $G_{D^\gamma_{< i} \cap \tau_{k'}}$ since 
			\begin{itemize}
				\item $k \in G_{D^\gamma_{\leq i}} \cap G_{\tau_k +1} = G_{D^\gamma_{\leq i} \cap (\tau_k +1)} \subseteq G_{\tau_k +1} \subseteq G_{\tau_{k'}}$;\footnote{The first equality is by Lemma~\ref{clcor}.}
				\item $k' \in G_{\tau_{k'}+1} \setminus G_{\tau_{k'}}$;
				\item $G_{D^\gamma_{< i} \cap \tau_{k'}} \subseteq G_{\tau_{k'}}$.
			\end{itemize}
			Now we prove by induction that for $\delta \in [\tau_{k'},\gamma]$:
			 \begin{equation} \label{fel} k \notin G_{D^\gamma_{< i} \cap \delta}(k')^{\pm 1} G_{D^\gamma_{< i} \cap \delta}. \end{equation}

			Since the sequence $\langle G_{D^\gamma_{< i} \cap \delta}\mid \delta \leq \gamma \rangle$ is obviously continuous we only have to handle the successor steps, i.e., if $\delta$ satisfies \eqref{fel}, $\delta \in D^\gamma_{< i}$, then \eqref{fel} holds for $\delta+1$ (instead of $\delta$), too.
		
			Now we recall the recursive construction of $G_{\delta+1}$. 
Since $\delta \in D^\gamma_{< i} \subseteq D^\gamma_{\leq i}$, and $D^\gamma_{< i}$, $D^\gamma_{\leq i}$ are closed in the sense defined just before Lemma~\ref{clcor} (and $D^\gamma_{< i} \cap \delta \subsetneq D^\gamma_{\leq i}$ as $\tau_k,\tau_{k'} \in D^\gamma_{\leq i} \setminus D^\gamma_{< i}$),
			we have that $D^\gamma_{< i} \cap \delta = D^\delta_{<\eta}$, $D^\gamma_{\leq i} \cap \delta = D^\delta_{<\zeta}$ for some $\eta < \zeta < \theta$.
			Let $\xi = i^\delta_k$, $\xi' = i^\delta_{k'}$, i.e., they denote the minimal ordinals such that $k \in G_{D^\delta_{\leq \xi}}$, $k' \in G_{D^\gamma_{\leq \xi'}}$.
			Observe that if $\xi \neq \xi'$ (say, $\xi' < \xi$), then $k' \in G_{D^\delta_{\leq \xi'}} \leq G_{D^\delta_{\leq \xi'} \cup \{\delta\}}$, while $k \notin G_{D^\delta_{\leq \xi'}}$ implies $k \notin G_{D^\delta_{\leq \xi'} \cup \{\delta\}}$ (by Lemma~\ref{clcor}), so $k$ and $k'$ are good fellows over $G_{D^\delta_{\leq \xi'} \cup \{\delta\}}$.
			Therefore $G_{D^\gamma_{< i} \cap (\delta+1)} \leq G_{D^\delta_{\leq \xi'} \cup \{\delta\}}$ implies that $k$ and $k'$ are good fellows over $G_{D^\gamma_{< i} \cap (\delta+1)}$, and we are done.
			
			So w.l.o.g.~$\xi = \xi'$.
			Now we are going to recall how $G_{D^\delta_{\leq \xi} \cup \{\delta\}}$ was constructed as an application  Lemma~\ref{shhe} to $K = G_{D^\delta_{\leq \xi}}$, $L = G_{D^\delta_{<\xi} \cup \{\delta\}}$, and we apply of \ref{H'} of Lemma~\ref{shhe}  for $k$, $k'$ (which are good fellows over $G_{D^\gamma_{< i} \cap \delta} = G_{D^\delta_{< \eta}}$ by the hypothesis) and $L' = G_{D^\delta_{< \eta} \cup \{\delta\}}$, $H'=G_{D^\delta_{< \eta}}$. 
			By Lemma~\ref{clcor} and $\xi > \eta$
			$$ L' \cap K = G_{D^\delta_{<\eta} \cup \{\delta\}} \cap G_{D^\delta_{\leq \xi}} = G_{D^\delta_{<\eta}} = H',$$
			so we can indeed apply \ref{H'}, and so get that $k$, $k'$ are good fellows over 
			$$L' = G_{D^\delta_{< \eta} \cup \{\delta\}} = G_{D^\gamma_{< i} \cap (\delta+1)},$$
			 and we are done.
		\end{PROOF}	
				To deal with the case 	$\tau_k = \tau_{k'}$, and so $\alpha_\sigma = \alpha_{\sigma'}$, we distinguish the following scenarios:
				
$\br$ If $l = l'$, then one can check that this together with $\alpha_\sigma = \alpha_{\sigma'}$ implies $b_\sigma = b_{\sigma'}$, $b'_{\sigma} = b'_{\sigma'}$, moreover, $k \neq k'$, so $a_\sigma \neq a_{\sigma'}$.
				
$\br$ If $l \neq l' \in T^\gamma_{<i, \alpha_\sigma}$, then without loss of generality,
we may assume that $b_\sigma$ and $b_{\sigma'}$ are not good fellows over $H$ (since otherwise we would be done).
				First note that $d_\sigma = d_{\sigma'} \in G_{D^\gamma_{< i^\gamma_{d_\sigma}}} \leq G_{D^\gamma_{< i}}$, and necessarily $l$, $l'$ are not good fellows over $H$.
				Therefore, $b_\sigma$, and $b_{\sigma'}$ are clearly $E^\gamma_{<i,\gamma}$-equivalent, hence a straightforward application of \ref{gfa} from Lemma~\ref{shhe} yields that $b'_\sigma$ and $b_{\sigma'}$ (in fact, $b_\sigma$ and $b'_{\sigma'}$ too) are good fellows over $H$.				 
						On the other hand, by the definition of $T^\gamma_{<i, \alpha_\sigma}$ we get that $l \neq l'$ are good fellows over $G_{D^\gamma_{< i} \cap \alpha_\sigma} = G_{D^\gamma_{\le i} \cap \alpha_\sigma} \cap G_{D^\gamma_{< i} } =K'_\sigma \cap H$, so are $b_\sigma$ and $b_{\sigma'}$, since $d_\sigma=d_{\sigma'}$, $y_{\sigma,0}$, $y_{\sigma,1}$ $\in G_{D^\gamma_{<i} \cap \alpha_\sigma}$ (by recalling the definition of $J \subseteq J^+$). 
						
						It remains to check \eqref{K'rel}, which can be proved by a very similar argument that is utilized in Claim~\ref{goodfella}, but we include the proof.
					\begin{claim}\label{claim49}
						Assume that $\alpha \in D^\gamma_{\leq i} \setminus D^\gamma_{<i}$, $g \in G_{D^\gamma_{<i}} \setminus G_\alpha$, and $k,k' \in G_{D^\gamma_{\le i} \cap (\alpha+1)} \setminus G_{\alpha}$. Then $kgk' \notin G_{D^\gamma_{<i}}$.
					\end{claim}
					\begin{PROOF}{Claim~\ref{claim49}}It suffices to prove that for each $\beta \in D^\gamma_{<i} \setminus \alpha$ for no $g \in G_{D^\gamma_{<i} \cap (\beta+1)} \setminus G_{D^\gamma_{<i} \cap \beta}$ do exist $k,k' \in G_{D^\gamma_{\le i} \cap (\alpha+1)} \setminus G_{\alpha}$ with $kgk' \in G_{D^\gamma_{<i} \cap (\beta+1)}$.
						
						Suppose not, so that $kgk' \in G_{D^\gamma_{<i} \cap (\beta+1)}$.
						Now in the same line of reasoning as in Claim~\ref{goodfella}, there are $\eta < \zeta < \theta$ with $D^\gamma_{<i} \cap \beta = D^\beta_{<\eta}$, and $D^\gamma_{\le i} \cap \beta = D^\beta_{< \zeta}$, and $k,k' \in G_{D^\beta_{<\zeta} \cap (\alpha+1)} \leq G_{D^\beta_{<\zeta}}$. Again, for some $\xi,\xi' \in [\eta, \zeta)$, 
\begin{itemize}
\item $k \in G_{D^\beta_{\le\xi}} \setminus  G_{D^\beta_{<\xi}}$.
\item $  k' \in G_{D^\beta_{\le\xi'}} \setminus  G_{D^\beta_{<\xi'}}$.
\end{itemize}
						Assume first that $\xi = \xi'$. Then it is enough to prove that 
						$kgk' \notin G_{D^\beta_{<\xi} \cup \{\beta\}}$, as 
						$$ G_{D^\gamma_{<i} \cap (\beta+1)} = G_{D^\beta_{<\eta} \cup \{\beta\}} \leq G_{D^\beta_{<\xi} \cup \{\beta\}}.$$
						But $G_{D^\beta_{\le \xi} \cup \{\beta\}}$ was given by Lemma~\ref{shhe} for $L = G_{D^\beta_{<\xi} \cup \{\beta\}}$, $K= G_{D^\beta_{\le \xi}}$, $H = G_{D^\beta_{<\xi}}$, and just apply (the parallel of)  \ref{par}.
						
						The case $\xi \neq \xi'$ is easier, then one has to appeal to the trivial part of \ref{par}.
					\end{PROOF}
	
	This completes the proof.
		\end{PROOF}

It follows that we may invoke Lemma~\ref{shhe}
with $H = G_{D^\gamma_{< i}}$,
		 $ K = G_{D^\gamma_{\le i}}$,
		 $L = G_{D^\gamma_{< i} \cup \{\gamma\}}$,
and the above $S = \{ (b_\sigma,b'_{\sigma}, a_\sigma, h_\sigma)\mid \sigma \in J \}$.
We then let $G_{D^\gamma_{\le i}}$ be the outcome $M^*$.
		
This completes the recursive construction of our group $G=G_\kappa$.
\subsection{Verification}Recalling \ref{underlying} and the system of mappings $\vec q=(q_0,q_1,q_2,q_3)$,
is easy to see that the following set is a club in $\kappa$:
$$C=\{\delta<\kappa\mid G_\delta=\delta\ \&\ \vec q[\delta] = {}^3G_\delta \times \{-1,1\}\}.$$

We now turn to show that $G$ is an $n$-Shelah group for $n=10120$.
		
		\begin{lemma} Let $Z\in[G]^\kappa$. Then $Z^{10120}=G$.
		\end{lemma}
		\begin{proof} By possibly thinning out (using the pigeonhole principle),
		we may assume the existence of some $i^\bullet<\theta$ such that $i_z=i^\bullet$ for all $z\in Z$.
		Set $A=\{ \tau_z\mid z\in Z\}$, so that $A\in [\kappa]^\kappa$.
		For each $\alpha\in A$, pick $z_\alpha\in Z$ such that $\tau_{z_\alpha}=\alpha$.
		
		Recalling the hypothesis of Theorem~\ref{thm41}, 
we now let $D$ be a club in $\kappa$ such that for every $\delta\in D$,
	for every $\beta\in A\setminus \delta$,
	for every $(\xi_0,\xi_1)\in\delta\times\delta$, 
	for every $i<\theta$, there are cofinally many $\alpha<\delta$ such that $\alpha\in A$ and
	$$c_0(\alpha,\beta)=\xi_0\ \&\ c_1(\alpha,\beta)=\xi_1 \ \&\ e(\alpha,\beta)>i.$$
		
\begin{claim} Let $g\in G$. Then $g$ is in $Z^{10120}$.
\end{claim}
\begin{proof} By \ref{underlying}, $G$ has underlying set $\kappa$, 
so $g$ is in fact an ordinal in $\kappa$, which we shall denote by $\xi_0$.
Pick $\delta\in C\cap D$ above $\xi_0$,
and then pick $\beta\in A$ above $\delta$.
As $\delta\in C$, $\xi_0 \in G_\delta$.
As $\delta\in D$, $\{ \bar\alpha\in A\cap\delta\mid e(\bar\alpha,\beta)>i^\bullet\}$ is cofinal in $\delta$,
so we fix $\bar\alpha\in A\cap\delta$ with $e(\bar\alpha,\beta)>i^\bullet$.
Set $h=z_{\bar\alpha}$, so that $h\in G_{\delta}$
and $i^\beta_{h} \geq e(\bar\alpha,\beta) > i^\bullet$. 

			Using \ref{fontt2} there exists $t \in T^{\beta}_{<i^\bullet+1, \delta}$
			 such that $z_\beta$ and $t$ are $E^{\beta}_{<i^\bullet+1, \delta}$-equivalent, 
			 and by \ref{fontt2} for  some $y_{0}, y_{1} \in G_{D^\beta_{< i^\bullet +1 } \cap \delta}$
			and  $\varp \in \{-1,1\}$, it is the case that
		$$z_\beta = y_{0} t^{\varp} y_{1},$$
		hence 
		$$ \max \{i_{t}, i^{\beta}_{y_{0}}, i^{\beta}_{y_{1}}\} \leq i^\bullet,$$
		and 
		\begin{equation} \label{y} y_{0}, y_{1} \in G_{D^\beta_{<i^\bullet +1}} \cap G_{\delta}. \end{equation}

Find a large enough $\gamma < \delta$ such that $y_0,y_1,h \in G_{\gamma+1}$.
Thus, $y_0,y_1,h \in G_{D^\beta_{\le i^\beta_h}\cap (\gamma+1)}$.
		Note that, again by \ref{fontt2}, $t \in  E^{\beta}_{<i^\bullet+1, \alpha}$ for every $\alpha \leq \delta$.
		
		As $\delta\in C$ and $h\in G_\delta$,
		it follows from \eqref{y} that we may find $\xi_1 < \delta$ such that
		$$(q_0(\xi_1), q_1(\xi_1), q_2(\xi_1), q_3(\xi_1)) = (y_{0}, y_{1}, h, \varp).$$

At this stage, we utilize the fact that $\delta\in D$, and pick $\alpha \in A \cap \delta$ above $\gamma$ such that
			$c_0(\alpha,\beta) = \xi_0$, $c_1(\alpha,\beta) = \xi_1$, and $e(\alpha,\beta) > i^\beta_h$.
Consider $i=e(\alpha,\beta)$, and note that $i>i^\beta_h>i^\bullet$.
			Clearly 
			$$y_0,y_1,h \in G_{D^\beta_{\le i^\beta_h } \cap (\gamma+1)} \subseteq G_{D^\beta_{< i} \cap \alpha},$$
			and by the subadditivity of $e$, it is the case that
			$$D^\alpha_{\le i_{z_\alpha}} = D^\alpha_{\le i^\bullet} \subseteq  D^\alpha_{\le e(\alpha,\beta)} = D^\beta_{\le e(\alpha,\beta)} \cap \alpha,$$
			thus $z_\alpha \in G_{D^\alpha_{\le i_{z_\alpha}}}$ easily implies $z_\alpha \in G_{D^\beta_{\le e(\alpha,\beta)}}$.
			
Finally, check that $b = z_\beta h$, $b' = z_\beta h z_\beta h$, $a = z_\alpha$ we have that $g^{-1} \varrho(ba,b'a) = \one$ holds in $G_{D^\beta_{\le i^\beta_h} \cup \{\beta\}}$, as $(l,k)=(t,h) \in J$ in \ref{constJ} in \ref{constJ+}. As $\varrho(ba,b'a)$ is a word of length $3320$ over $ba$, $b'a$, and $z_\beta, h, z_\alpha \in Z$, we get that
$$ g = \varrho(z_\beta h z_\alpha, z_\beta h z_\beta h z_\alpha) \in Z^{9720 + 400}.$$
\end{proof}		 
As $g\in G$ was arbitrary,
the preceding claim establishes that $Z^{10120}=G$.
\end{proof}
	
\begin{lemma}
\begin{enumerate}
\item 	$G$ admits no $T_1$ topology other than the discrete topology;
\item  $G \setminus \{\one\}$ is a nonalgebraic unconditionally closed set (i.e., closed in each Hausdorff group topology).
\end{enumerate}
\end{lemma}
\begin{proof}
(1) This is a standard consequence of the malnormality of $G_\gamma$'s ($\gamma < \kappa$). Suppose on the contrary and fix a nondiscrete Hausdorff topology on $G$, and fix $g \in G$ distinct from $\one$. Then $U_0 = G \setminus \{g\}$ is open, so there is an open neighborhood $U_1$ of $\one$ for which $(U_1)^n \subseteq U_0$,
where $n$ is integer for which $G$ is $n$-Shelah.
 Now if $|U_1| = \kappa$, then $U_1^n = G$, which is a contradiction, so it must be the case that $|U_1| < \kappa$. But then $U_1 \subseteq G_\gamma$ for some $\gamma < \kappa$. Now $G_\gamma \mal G_{\gamma+1}$, so for some $h \notin G_\gamma$ we have $hU_1h^{-1} \cap U_1 = \{\one\}$, an open neighborhood of $\one$, contradicting our assumption that the topology is nondiscrete.
	
(2) We need to show that for no system $w_i$ ($i \in I$) of words over $G \cup \{x\}$ (where $x$ is an abstract variable outside $G$) do we have 
	$$G \setminus \{\one\} = \bigcap\nolimits_{i \in I} \{g \in G \mid \ f_{w_i}(g) = \one\},$$
	where the value of $f_{w_i}(g) \in G$ is given by substituting each occurrence of the symbol $x$  in $w_i \in{}^{<\omega}(G \cup \{x\})$ with $g$, and calculating the value in $G$.
	It is easy to see that it suffices to prove that for no such word $w$ does the following equation holds true:
	\begin{equation} \label{wimpossible} G \setminus \{\one\} =  \{g \in G \mid \ f_{w}(g) = \one\}.\end{equation}

	Suppose that $w$ satisfies \eqref{wimpossible}, and fix a finite set $B \in [G]^{<\aleph_0}$ with $w \in{}^{<\omega}(B \cup \{x\})$. Now there exist $\gamma< \kappa$, $i < \theta$ such that
	$$B \subseteq G_{D^\gamma_{\leq i} \cup \{\gamma\}},$$
	so for each $g \in G_{D^\gamma_{\leq i} \cup \{\gamma\}}$ that is not the identity $f_w(g) =  \one$.
	
	We are going to prove that $G_{D^\gamma_{\leq i} \cup \{\gamma\}}$ is topologizable (with a nondiscrete $T_1$ topology), which will imply that $G_{D^\gamma_{\leq i} \cup \{\gamma\}} \setminus \{\one\}$ is closed, contradicting that the topology was nondiscrete.
	
To this end, it is enough to argue that there exists a sequence $\langle N^*_k \mid \ k \in \omega \rangle$ of normal subgroups of $G_{D^\gamma_{\leq i} \cup \{\gamma\}}$ such that for each $k$ do $N^*_{k+1} \leq N^*_k$, $\bigcap_{k \in \omega} N^*_k = \{\one\}$ and $\{\one\} \lneq N^*_k$ hold.
	Now recall how $G_{D^\gamma_{\leq i} \cup \{\gamma\}}$ was constructed in subsection~\ref{reccons} (appealing to Lemma \ref{shhe} there):
	$$G_{D^\gamma_{\leq i} \cup \{\gamma\}} = (G_{D^\gamma_{< i} \cup \{\gamma\}} \ast_{G_{D^\gamma_{< i}}}  G_{D^\gamma_{\leq  i}}) / N,$$
	where $N$ was the normal closure of $\{ h_\sigma^{-1} \varrho(b_\sigma a_\sigma,b_\sigma'a_\sigma)\mid \sigma \in J\}$ ($J$ is from Definition \ref{relde}).
	Let $N_0$ denote this $N$.
	Observe that it is enough to define a sequence $\langle N_k \mid \ k \in \omega \setminus \{0\}\rangle$ of normal subgroups in $G_{D^\gamma_{< i} \cup \{\gamma\}} \ast_{G_{D^\gamma_{< i}}}  G_{D^\gamma_{\leq  i}}$ that satisfies  $N_{k+1} \leq N_k$ for $k\geq 1$, $\bigcap_{k \in \omega} N_k = N_0$ and $N_0 \lneq N_k$.
	
	We define the sequence $\langle n_\ell\mid \ell \in \omega\rangle$ as $n_\ell = 6640^\ell$,
	let $\varrho_\ell(x,y) = \varrho(x^{n_\ell},y^{n_\ell})$ (so that $\varrho_0 = \varrho$), and 
	$$R_k:=  \{ h_\sigma^{-1}\varrho_0(b_\sigma a_\sigma,b'_\sigma a_\sigma),  \varrho_\ell(b_\sigma a_\sigma,b'_\sigma a_\sigma) \mid \ell \geq k,~ \sigma \in J\}.$$
	Set $N_k$ to be the normal closure of $R_k$.
	Now the following facts will complete the proof:
	\begin{itemize}
		\item if $\sigma \in J$, $k>0$, then 
		$$ G_{D^\gamma_{< i} \cup \{\gamma\}} \ast_{G_{D^\gamma_{< i}}}  G_{D^\gamma_{\leq  i}} \models \ \varrho_k(b_\sigma a_\sigma,b_\sigma'a_\sigma) \in N_k \setminus N_0,$$
		\item $R_1$ satisfies $C'(\frac{1}{10})$, moreover, if the group element $g \in G_{D^\gamma_{< i} \cup \{\gamma\}} \ast_{G_{D^\gamma_{< i}}}  G_{D^\gamma_{\leq  i}}$ has a canonical representation of length $<n_k$ for some $k\geq 1$, and $g \notin N_0$, then $g \notin N_k$.\qedhere
	\end{itemize}
\end{proof}

\begin{corollary}\label{cor513}  For every infinite regular cardinal $\lambda$, there exists a Shelah group of size $\lambda^+$.
\end{corollary}
\begin{proof} Invoke Theorem~\ref{thm41} with the pair $(\kappa,\theta)=(\lambda^+,\lambda)$,
using Theorem~\ref{successorofregular}.
\end{proof}
\begin{corollary}  For every regular uncountable cardinal $\kappa$, if $\square(\kappa)$ holds, then there exists a Shelah group of size $\kappa$.
\end{corollary}
\begin{proof} By Theorem~\ref{thm41} together with Theorem~\ref{generalcase}.
\end{proof}

\begin{corollary}\label{cor514}  In G\"odel's constructible universe, for every regular uncountable cardinal $\kappa$, the following are equivalent:
\begin{itemize}
\item there exists a Shelah group of size $\kappa$;
\item $\kappa$ is not weakly compact.
\end{itemize}
\end{corollary}
\begin{proof} By \cite{jensen72}, in G\"odel's constructible universe, every regular uncountable $\kappa$ is either weakly compact, or $\square(\kappa)$ holds. 
Therefore it suffices to prove that weakly compact cardinals do not carry a Shelah group. 
To this end, suppose that $G$ is an $n$-Shelah group with underlying set $\kappa$. Clearly, the group operation gives rise to a system $\langle f_{j}: [\kappa]^n \to \kappa \mid j<n^n \rangle$ satisfying that for every $Z\in[\kappa]^\kappa$, $$\bigcup_{j<n^n} f_j``[Z]^n = \kappa.$$

For every $j<n^n$, define $g_j:[\kappa]^n\rightarrow(n^n+1)$ via:
$$g_j(u):=\begin{cases}
f_j(u),&\text{if }f_j(n)<n^n;\\
n^n,&\text{otherwise}.
\end{cases}$$

 Define $c:[\kappa]^n\rightarrow{}^{n^n}(n^n+1)$ via:
 $$c(u)=\langle g_j(u)\mid j<n^n\rangle.$$

Finally, applying the weak compactness of $\kappa$, we find a set $Z \in [\kappa]^\kappa$ that is $c$-homogeneous.
Pick an $m\in n^n+1$ distinct from all elements of $c``[Z]^n$, since it is a sequence of length $n^n$.
Then $m$ is not in $Z^n$, which was supposed to be whole of $\kappa$, being the underlying set of some $n$-Shelah group.
\end{proof}

\section{Acknowledgments}

The first author was supported by the Excellence Fellowship Program for International Postdoctoral Researchers of The Israel Academy of Sciences and Humanities, and by the National Research, Development and Innovation Office – NKFIH, grants no. 124749, 129211. 
The second author was partially supported by the Israel Science Foundation (grant agreement 203/22)
and by the European Research Council (grant agreement ERC-2018-StG 802756).

The results of this paper were presented by the first author at the \emph{Toronto Set Theory seminar} in May 2023.
We thank the organizers for the opportunity to speak and the participants for their feedback.


\begin{thebibliography}{MMR09}

\bibitem[Adi06]{adian2006classifications}
Sergei~I Adian.
\newblock Classifications of periodic words and their application in group theory.
\newblock In {\em Burnside Groups: Proceedings of a Workshop Held at the
  University of Bielefeld, Germany June--July 1977}, pages 1--40. Springer, 2006.

\bibitem[Ban22]{banakh2022nonpolybounded} Taras Banakh.
\newblock A non-polybounded absolutely closed $36$-shelah group, 2022.

\bibitem[Ber06]{MR2239037}
George~M. Bergman.
\newblock Generating infinite symmetric groups.
\newblock {\em Bull. London Math. Soc.}, 38(3):429--440, 2006.

\bibitem[BTV12]{MR2959420}
Valery Bardakov, Vladimir Tolstykh, and Vladimir Vershinin.
\newblock Generating groups by conjugation-invariant sets.
\newblock {\em J. Algebra Appl.}, 11(4):1250071, 16, 2012.

\bibitem[Cor22]{MR4428866}
Samuel~M. Corson.
\newblock J\'{o}nsson groups of various cardinalities.
\newblock {\em Proc. Amer. Math. Soc.}, 150(7):2771--2775, 2022.

\bibitem[COV23]{corson2023steep}
Samuel~M. Corson, Alexander Olshanskii, and Olga Varghese.
\newblock Steep uncountable groups, 2023.

\bibitem[dC06]{cornulier2006strongly}
Yves de~Cornulier.
\newblock Strongly bounded groups and infinite powers of finite groups.
\newblock {\em Comm. Algebra}, 34(7):2337--2345, 2006.

\bibitem[DG05]{MR2122432}
Manfred Droste and R\"{u}diger G\"{o}bel.
\newblock Uncountable cofinalities of permutation groups.
\newblock {\em J. London Math. Soc. (2)}, 71(2):335--344, 2005.

\bibitem[DH05]{MR2154425}
Manfred Droste and W.~Charles Holland.
\newblock Generating automorphism groups of chains.
\newblock {\em Forum Math.}, 17(4):699--710, 2005.

\bibitem[DHU08]{MR2418802}
Manfred Droste, W.~Charles Holland, and Georg Ulbrich.
\newblock On full groups of measure-preserving and ergodic transformations with
  uncountable cofinalities.
\newblock {\em Bull. Lond. Math. Soc.}, 40(3):463--472, 2008.

\bibitem[Dow20]{MR4095507}
Philip~A. Dowerk.
\newblock Strong uncountable cofinality for unitary groups of von {N}eumann
  algebras.
\newblock {\em Forum Math.}, 32(3):773--781, 2020.

\bibitem[DT09]{MR2645225}
Manfred Droste and John~K. Truss.
\newblock Uncountable cofinalities of automorphism groups of linear and partial
  orders.
\newblock {\em Algebra Universalis}, 62(1):75--90, 2009.

\bibitem[EHR65]{MR202613}
P.~Erd\H{o}s, A.~Hajnal, and R.~Rado.
\newblock Partition relations for cardinal numbers.
\newblock {\em Acta Math. Acad. Sci. Hungar.}, 16:93--196, 1965.

\bibitem[FR17]{paper27}
David {Fernandez-Breton} and Assaf Rinot.
\newblock Strong failures of higher analogs of {H}indman's theorem.
\newblock {\em Trans. Amer. Math. Soc.}, 369(12):8939--8966, 2017.

\bibitem[G{\"o}d40]{MR2514}
Kurt G{\"o}del.
\newblock {\em The {C}onsistency of the {C}ontinuum {H}ypothesis}.
\newblock Annals of Mathematics Studies, no. 3. Princeton University Press,
  Princeton, N. J., 1940.

\bibitem[Hes79]{hesse1979topologisierbarkeit}
Gerhard Hesse.
\newblock Zur topologisierbarkeit von gruppen.
\newblock {\em Dissertation, Univ. Hannover}, 1979.

\bibitem[Hin74]{MR0349574}
Neil Hindman.
\newblock Finite sums from sequences within cells of a partition of {$N$}.
\newblock {\em J. Combinatorial Theory Ser. A}, 17:1--11, 1974.

\bibitem[HJ74]{MR336705}
A.~Hajnal and I.~Juh\'{a}sz.
\newblock On hereditarily {$\alpha $}-{L}indel\"{o}f and {$\alpha $}-separable
  spaces. {II}.
\newblock {\em Fund. Math.}, 81(2):147--158, 1973/74.

\bibitem[IR22]{paper53}
Tanmay Inamdar and Assaf Rinot.
\newblock Was {U}lam right? {I}{I}: {S}mall width and general ideals.
\newblock Submitted March 2022.
\newblock \verb"http://assafrinot.com/paper/53".

\bibitem[Jen72]{jensen72}
R.~Bj{\"o}rn Jensen.
\newblock The fine structure of the constructible hierarchy.
\newblock {\em Ann. Math. Logic}, 4:229--308; erratum, ibid. 4 (1972), 443,
  1972.
\newblock With a section by Jack Silver.

\bibitem[J{\'{o}}n72]{MR0345895}
Bjarni J{\'{o}}nsson.
\newblock {\em Topics in universal algebra}.
\newblock Lecture Notes in Mathematics, Vol. 250. Springer-Verlag, Berlin-New
  York, 1972.

\bibitem[KT05]{KlyachkoTrofimov}
Anton~A. Klyachko and Anton~V. Trofimov.
\newblock The number of non-solutions of an equation in a group.
\newblock 8(6):747--754, 2005.

\bibitem[LR23]{paper36}
Chris {Lambie-Hanson} and Assaf Rinot.
\newblock Knaster and friends {III}: {S}ubadditive colorings.
\newblock {\em J. Symbolic Logic}, 2023.
\newblock Accepted June 2022.

\bibitem[LS77]{lyndon1977combinatorial}
Roger~C. Lyndon and Paul~E. Schupp.
\newblock {\em Combinatorial group theory}.
\newblock Ergebnisse der Mathematik und ihrer Grenzgebiete, Band 89.
  Springer-Verlag, Berlin-New York, 1977.

\bibitem[MMR09]{MR2520386}
V.~Maltcev, J.~D. Mitchell, and N.~Ru\v{s}kuc.
\newblock The {B}ergman property for semigroups.
\newblock {\em J. Lond. Math. Soc. (2)}, 80(1):212--232, 2009.

\bibitem[MN90]{macpherson1990subgroups}
H~Dugald Macpherson and Peter~M Neumann.
\newblock Subgroups of infinite symmetric groups.
\newblock {\em Journal of the London Mathematical Society}, 2(1):64--84, 1990.

\bibitem[Moo06]{Moore}
Justin~Tatch Moore.
\newblock A solution to the {$L$} space problem.
\newblock {\em J. Amer. Math. Soc.}, 19(3):717--736 (electronic), 2006.

\bibitem[Ols80]{MR571100}
A.~Ju. Olshanskii.
\newblock An infinite group with subgroups of prime orders.
\newblock {\em Izv. Akad. Nauk SSSR Ser. Mat.}, 44(2):309--321, 479, 1980.

\bibitem[Ols12]{ol2012geometry}
A~Yu Olshanskii.
\newblock {\em Geometry of defining relations in groups}, volume~70.
\newblock Springer Science \& Business Media, 2012.

\bibitem[Ros09]{MR2503307}
Christian Rosendal.
\newblock A topological version of the {B}ergman property.
\newblock {\em Forum Math.}, 21(2):299--332, 2009.

\bibitem[RR07]{MR2332091}
\'{E}ric Ricard and Christian Rosendal.
\newblock On the algebraic structure of the unitary group.
\newblock {\em Collect. Math.}, 58(2):181--192, 2007.

\bibitem[RT13]{paper14}
Assaf Rinot and Stevo Todorcevic.
\newblock Rectangular square-bracket operation for successor of regular
  cardinals.
\newblock {\em Fund. Math.}, 220(2):119--128, 2013.

\bibitem[RZ23a]{paper52}
Assaf Rinot and Jing Zhang.
\newblock Complicated colorings, revisited.
\newblock {\em Ann. Pure Appl. Logic}, 174(4):Paper No. 103243, 2023.

\bibitem[RZ23b]{paper45}
Assaf Rinot and Jing Zhang.
\newblock Strongest transformations.
\newblock {\em Combinatorica}, 2023.
\newblock To appear. \verb"https://doi.org/10.1007/s00493-023-00011-0".

\bibitem[She80]{Sh:69}
Saharon Shelah.
\newblock On a problem of {K}urosh, {J}\'{o}nsson groups, and applications.
\newblock In {\em Word problems, {II} ({C}onf. on {D}ecision {P}roblems in
  {A}lgebra, {O}xford, 1976)}, volume~95 of {\em Studies in Logic and the
  Foundations of Mathematics}, pages 373--394. North-Holland, Amsterdam-New
  York, 1980.

\bibitem[She20]{MR4186458}
Saharon Shelah.
\newblock Density of indecomposable locally finite groups.
\newblock {\em Rend. Semin. Mat. Univ. Padova}, 144:253--270, 2020.

\bibitem[Sip06]{sipacheva2006consistent}
Ol'ga~V Sipacheva.
\newblock Consistent solution of {M}arkov's problem about algebraic sets.
\newblock {\em arXiv preprint math/0605558}, 2006.

\bibitem[Tod07]{todorcevic_book}
Stevo Todorcevic.
\newblock {\em Walks on ordinals and their characteristics}, volume 263 of {\em
  Progress in Mathematics}.
\newblock Birkh\"auser Verlag, Basel, 2007.

\bibitem[Tol06a]{MR2266526}
V.~A. Tolstykh.
\newblock Infinite-dimensional general linear groups are groups of finite
  width.
\newblock {\em Sibirsk. Mat. Zh.}, 47(5):1160--1166, 2006.

\bibitem[Tol06b]{MR2241973}
Vladimir Tolstykh.
\newblock On the {B}ergman property for the automorphism groups of relatively
  free groups.
\newblock {\em J. London Math. Soc. (2)}, 73(3):669--680, 2006.

\bibitem[TZ12]{MR2982769}
Simon Thomas and Jind\v{r}ich Zapletal.
\newblock On the {S}teinhaus and {B}ergman properties for infinite products of
  finite groups.
\newblock {\em Confluentes Math.}, 4(2):1250002, 26, 2012.

\end{thebibliography}
\end{document}